\documentclass[12pt]{article}
\usepackage[a4paper]{geometry}
\usepackage{amsthm}

\usepackage{amssymb}
\usepackage{amsmath}
\usepackage{eufrak}

\newtheorem{theorem}{Theorem}

\newtheorem{remark}[theorem]{Remark}

\newtheorem{corollary}[theorem]{Corollary}
\newtheorem{lemma}[theorem]{Lemma}

\newtheorem{definition}[theorem]{Definition}
\newtheorem{question}{Question}

\newtheorem{proposition}[theorem]{Proposition}

\begin{document}
This arXiv submission was split into two parts. The second part begins at page 25
and the abstract of the second part appears on page 25.
\pagebreak
\def\F{{\mathbb F}}
\title{ On some connections between braces and pre-Lie rings outside of the context of Lazard's correspondence}
\author{ Agata Smoktunowicz}
\date{ }
\maketitle

\begin{abstract}  
Let $A$ be a brace of cardinality $p^{n}$ for some prime number $p$.   Let $k$ be a natural number. Denote $ann(p^{k})=\{a\in A: p^{k}a=0\}$.  
It is shown that,  under some additional assumptions, the quotient brace $A/ann(p^{2k})$ 
 is obtained by an  algebraic formula from the brace $p^{k}A$, and that this formula only depends on the additive group of brace $A$. Moreover,  $p^{k}A$ is a strongly nilpotent brace. 
 This result is used to describe the structure of some  types of braces.

\end{abstract} 
In 2005, Rump introduced the notion of a brace in order to describe all non-degenerate involutive set-theoretic solutions of the Yang-Baxter equation \cite{rump}. 
 Braces are similar to rings, and they were used to develop new methods  to describe involutive set-theoretic solutions of Yang-Baxter equation. Braces also have connections to quantum integrable systems, Hopf-Galois extensions, group theory, homological algebras  and Lie algebras.

The structure of  braces of cardinality $p^{n}$ for a prime $p$  and a natural $n$ with $p$ larger than $n$  is linked to pre-Lie rings and is relatively well understood \cite{Rump, passage, Senne}. 
 On the other hand, some very unusual results hold for braces of cardinality $p^{n}$ with $p<n$; for example, in \cite{Bellester}
 examples of right nil and not right nilpotent braces of cardinality $2^{5}$ were found.

 In this paper we  investigate the stucture of braces 
 of cardinality $p^{n}$ with $n$ larger than $p$. 
We obtain a pathway from  right nilpotent braces to such  braces, 
 under some additional assumptions.   The main motivation of our result   is that 
 the structure of right nilpotent braces seems to be better understood, and it might be  possible to connect right nilpotent braces with other algebraic structures, such as for example pre-Lie rings, Lie-rings,  or Lie affgebras introduced in \cite {an}. Some of these  connections with pre-Lie rings are described in \cite{paper1, paper3}.

The right nilpotency of braces and skew braces
has been investigated by various authors.  In some cases the structure of right nilpotent braces is already known.
  In \cite{KS}, Kurdachenko and Subbotin developed  new methods to describe the structure of  braces with small nilpotency index.

 Right nilpotent braces were introduced by Rump in \cite{rump}. Such braces are closely connected to set-theoretic solutions of a finite multipermutation level \cite{gateva}.  
 In \cite{BA3}, Bachiller developed a method for constructing  all right nilpotent braces,  and showed that all braces of cardinality $p^{3}$ for a prime number $p$ are right nilpotent.
  In \cite{Dora1}, Pulji{\' c} showed that all braces of cardinality $p^{4}$ for a prime number $p$ (except of one type)  are right nilpotent.
  Note also  that interesting examples of braces of cardinality $p^{4}$ were  constructed in \cite{DB}.
  Right nilpotent braces and their connections with set-theoretic solutions of the Yang-Baxter equation were  investigated in \cite{ TB, Cedo, cjo, gateva,  LV}. 
 Notice that an important class of set-theoretic solutions of finite multipermutation level is closely related to right nilpotent braces \cite{gateva}, as such solutions are obtained from right nilpotent braces when using the usual formula $r(x,y)=(x\circ y-x, z\circ x  )$ where $z$ is the inverse element of  $(x\circ y-x)$ in the multiplicative group of the brace (and finite solutions come from finite braces \cite{CGS}).
 The related concept of retraction of set-theoretic solutions was introduced by Etingof, Shedler and Soloviev (the retraction of set-theoretic solution corresponds to division of the brace by the socle and is related to nilpotency) \cite{Etingof}.
 Note that for non-involutive solutions there is also a retraction \cite{biracks}.

It is also known that every brace  which is both left nilpotent and right nilpotent is strongly nilpotent \cite{Engel}, so the product of any $n$ elements in this brace under operation $*$ is zero, for some $n$.
$ $

 As mentioned by Vendramin in \cite{49}, braces and skew braces of cardinality $p^n$ can
be viewed as building blocks of other braces. It was shown in \cite{42} that the
multiplicative group of a brace and subbraces which are its Sylow’s subgroups (which
have cardinality $p^n$) determine immediately the whole brace. 
 In \cite{Jespers}, Jespers, Van Antwerpen and Vendramin introduced  an analogon  Wedderburn-Artin theorem for braces and skew braces and introduced the radical of a given brace. The general approach in \cite{Jespers}  is  to describe the radical  $J(A)$ of a brace $A$ and the semisimple brace $A/J(A)$. Therefore the description of radical braces is important for the structure of general  braces.  Observe that all braces of cardinality $p^{n}$ with prime $p$ are radical, since by the result of Rump \cite{rump} they are left nilpotent.
 Recall that skew braces were introduced in \cite{GV} and they share many common properties with braces.

 Because braces are a generalisation of Jacobson radical rings they can be investigated by applying methods from noncommutative ring theory, and we will use these methods in this paper.
  Methods from noncommutative ring theory were  used in several other papers on related  topics \cite{TB, cjo, Rio, CO,  Facchini, KS}. 
 
The aim of this paper is to prove  the following result:

\begin{theorem}\label{main}
Let $A$ be a brace of cardinality $p^{n}$ for some prime number $p$. Suppose that for  $i=1,2,\ldots $ and all $a,b\in A$ we have 
 \[a*(a*(\cdots *a*b))\in pA, a*(a*(\cdots *a*ann(p^{i})))\in ann(p^{i-1})\]  where $a$ appears less than $\frac {p-1}4$ times in this expression. 
 Let $k$ be a natural number. Then the factor  brace $A/ann(p^{2k})$ 
 is obtained by an  algebraic formula from the brace $p^{k}A$ and that this formula only depends on the additive group of brace $A$.
\end{theorem}

 Recall also that if a brace $A$ satisfies the assumptions of Theorem \ref{main} then $ann(p^{i})$ is an ideal in brace $A$ for each $i$ by \cite{paper1}.

 Note that the brace $p^{k}A$ in Theorem \ref{main} is strongly nilpotent \cite{paper1}. That means that the product of any $n$ elements, for some $n$, with any distribution of brackets is zero (with the operation $a*b=a\circ b-a-b$).  
 This makes the structure of such a brace  easier to describe than the structure of a general brace.

 Recall that braces are related to many algebraic structures, for example to Jacobson-radical rings, with two-sided braces being exactly Jacobson radical rings.
   Braces are in one-to-one correspondence with braided groups with involutive braiding operators \cite{gateva}, a structure which was used to decribe 
 set-theoretic solutions of the Yang-Baxter equation since 1999.   It was shown by Rump \cite{rump} that the structure groups of braces introduced in \cite{Etingof} by Etingof, Schedler and Soloviev are braces. 
 Another connection exists with the algebraic number theory area of Hopf-Galois extensions. It is known that braces and skew braces are in  correspondence with Hopf-Galois extensions, with the only difference being that the ennumeration of elements matters in Hopf-Galois extensions; but it does not matter in isomorphic braces, so one brace may correspond to several Hopf-Galois extensions. 

 Braces  are also connected   to many other topics in algebra and mathematical physics,  
 for example braces are exactly bijective 1-cocycles on abelian groups \cite{Rio} and hence have connections with homological algebra and group theory. 
There are also strong connections to Lie algebras, cohomology of groups  and  cohomology of algebras  \cite{cs}.

The involutive set-theoretic  solutions of the Yang-Baxter equation  have connections with quantum integrable systems, and they allow us to construct quantum-integrable sysytems \cite{Doikou}. The situation is not the same for non-involutive solutions,  however non-involutive solutions appear in  virtual Knot theory and noncommutative algebraic geometry in connections with physics within the research of Majid. Recently connections between braces and flat mainfolds were discovered in \cite{crystal}. 

 In this paper we only consider involutive solutions of the YBE. All such solutions are described by braces by \cite{rump} and they all have connections with quantum integrable systems, as shown in \cite{doikou}.   Notice that the quantum integrable systems associated to the quantum groups are of a different type to the quantum integrable systems associated to braces and set-theoretic solutions of the Yang-Baxter equation.
 Braces and skew-braces are also connected to reflection and Pentagon equations \cite{pent, cov, doikou}.

\section{Background information } \label{b}

 We will use the same notation as in \cite{passage, asas, paper1}. We recall this notation below for the convenience of the reader:
A set $A$ with binary operations $+$ and $* $ is a {\em  left brace} 
  \[(a+b+a*b)* c=a* c+b* c+a* (b* c), \space  a* (b+c)=a* b+a* c,\]
for all $a, b, c\in A$; moreover  $(A, \circ )$ is a group, where we define $a\circ b=a+b+a* b$ and $(A, +)$ is an abelian group.
In what follows we will use the definition in terms of the operation `$\circ $' presented in \cite{cjo} (see \cite{Rump}
for the original definition): a set $A$ with binary operations of addition $+$ and multiplication $\circ $ is a brace if $(A, +)$ is an abelian group, $(A, \circ )$ is a group and for every $a,b,c\in A$
\[a\circ (b+c)+a=a\circ b+a\circ c.\]
  All braces in this paper are left braces, and we will just call them braces.
 
Let $(A, +, \circ )$ be a brace.  Recall that $I\subseteq A$ is an ideal in $A$ if for
$i,j\in I$, $a\in A$ we have $i+j\in I, i-j\in I, i*a, a*i\in I$ where $a*b=a \circ b-a-b$.

  Recall that if $I$ is an ideal in the brace $A$ then the factor brace $A/I$ is well defined.
The elements of the brace $A/I$ are cosets $[a]_{I}:= a + I=\{a+i:i\in I\}$ where $a \in A$, which we will simply denote by $[a]_{I}$, so $[a]_{I} =[b]_{I}$ if and only if $a-b\in I$.
 Recall that for $a\in A$ and a natural number $i$, $ia$ denotes the sum of $i$-copies of $a$.

Let $A$ be a brace and $a\in A$ then we denote $a^{\circ j}=a\circ a\cdots \circ a$
  where $a$ appears $j$ times in the product on the right hand side.

Let $A$ be a brace. Let $B, C\subseteq A$, then by $B*C$ we define the additive subgroup of $A$ generated by 
 elements $b*c$ where $b\in B$, $c\in C$.

 Recall that $A^{1}=A$ and inductively $A^{i+1}=A*A^{i}$ is the left nilpotent series of $A$. This series was introduced in \cite{Rump} by Wolfgang  Rump, and 
 he showed that if $A$ is a brace of cardinality $p^{n}$ for a prime number $p$ and a natural number $n$ then $A^{n+1}=0$. 

\subsection{Pullback}\label{pullback}
In this section we slightly modify the definition of a pullback from \cite{asas}. We use the same approach as in \cite{paper1}. 
Let $A$ be a brace of cardinality $p^{n}$ for some prime number $p$ and some natural number $n$. 
Fix a natural number $k$.  
For $a\in p^{k}A$ let $\wp^{-1}(a)$ denote an element $x\in A$ such that $p^{k}x=a$. Such an element may not be uniquely determined in $A$, but we can fix for every $a\in p^{k}A$ such an element $\wp^{-1}(a)\in A$.
Notice that $p^{k}(\wp^{-1}(a))=p^{k}x=a$.

 For $A$ as above, $ann(p^{i})=\{a\in A:p^{i}a=0\}$. 
 We recall a result from \cite{paper1}.
\begin{lemma}\label{33} Let $A$ be a brace and let $I$ be an ideal in $A$. Let $A/I$ be defined as above. Assume that $ann(p^{k})\subseteq I$. 
Let $\wp^{-1} : p^{k}A \rightarrow A$ be defined as above. Then, for $a, b\in p^{k}A$ we have
$[\wp^{-1}(a)]_{I} + [\wp^{-1}(b)]_{I} = [\wp^{-1}(a + b)]_{I}.$ 
This implies that for any integer $m$ we have
 $[m \wp^{-1}(a)]_{I} = [\wp^{-1}(ma)]_{I}.$
\end{lemma}

 In this paper we will assume that 
\[\wp^{-1}(a)=(\rho ^{-1} )^{k}(a)=\rho^{-1}( \cdots (\rho^{-1}(a)\cdots )),\] where $\rho ^{-1}:pA\rightarrow A$ is a given function $\rho ^{-1}:pA\rightarrow A$ such that $p\cdot \rho^{-1}(x)=x$, for $x\in pA$.
 Observe that
\[\rho ^{-1}(a+b)-\rho^{-1}(a)-\rho ^{-1}(b)\in ann(p),\]
for all $a, b\in A$.
 Lemma \ref{33} holds for function $\rho ^{-1}$ when we assume that  $k=1$.

\subsection{Pseudobraces} 
In this section we  introduce pseudobraces. All braces are pseudobraces.

\begin{definition}\label{pseudobrace}
Let $A$ be a set, and let $+$ and $\circ $ be binary operations of $A$. Let $a,b\in A$. We denote $a\circ b=a*b-a-b$.  
We denote  $a^{1}=a$ and inductively $a^{\circ  i+1}=a\circ a^{\circ  i}$.
We say that $A$ is a pseudobrace if the following properties hold:
\begin{enumerate} 
\item $(A,+)$ is an abelian group of cardinality $p^{n}$ for some prime number $p$ and some natural number $n$. 
\item For every $a\in A$ there is an element in $A$, denoted as $a^{\circ  -1}$, such that $a\circ a^{\circ -1}=0$, where $0$ is the identity in the group $(A,+)$.
 We denote $a^{\circ -2}=a^{\circ -1}\circ a^{\circ -1}$, $a^{\circ -3}=a^{\circ -1}\circ a^{\circ -2}$, etc..
\item For all integers $i,j$ and every $a,b\in A$, $(a^{\circ  i}\circ a^{\circ j})\circ b=a^{ \circ i+j}\circ b$ 
 and $a^{\circ i}\circ a^{\circ j}=a^{\circ i+j}$, $a^{circ i}\circ b=a\circ (a\circ \cdots a\circ b))$ where $a$ appears $i$ times on the right-hand side of this equation. 
\item $A^{n+1}=0$ where $A^{1}=A$ and inductively $A^{i+1}=A*A^{i}$.  
\item $a*(b+c)=a*b+a*c$ for all $a,b,c\in A$.
\end{enumerate}
\end{definition}

$ $

 The group $(A, +)$ is called the additive group of pseudobrace $A$.

$ $

For a pseudobrace $A$ we define $p^{i}A=\{p^{i}a:a\in A \}$, $ann(p^{i})=\{a\in A:p^{i}A=0\}$.

Pseudobraces were introduced in \cite{paper3} in the context of generalised group of flows. Every brace is a pseudobrace. Our main result (Theorem \ref{71}) holds both for braces and for pseudobraces with the same proof, hence we consider the more general version of pseudobrace.
 Readers who prefer to consider only pseudobraces can, in what follows, at each place  use the word 'brace' instead of 'pseudobrace'. 
\subsection{ Operation $\odot $ for braces and pseudobraces}

In this section we define operation $\odot $ for braces and pseudobraces.

Let $A$ be a pseudobrace of cardinality $p^{n}$ and let $k$ be a natural number. Note that if $ann(p^{2k})=\{a\in A: p^{2k}a=0\}$  is an ideal in the pseudobrace $A$ then the factor brace $A/ ann(p^{2k})$ is well defined.
The elements of the pseudobrace $A/I$ are cosets $[a]_{I}:= a + I=\{a+i:i\in I\}$ where $a \in A$. In this section we denote for $a\in a$,  $[a]_{ann(p^{2})}=[a]$.
  
The following result is a generalisation of Theorem $8.1$, \cite{asas}.
\begin{theorem}\label{odot}
Let $(A, +, \circ )$ be a pseudobrace of cardinality $p^{n}$ for some prime number $p$ and some natural number $n$. Let $k$ be a natural number such that $p^{k(p-1)}A=0$. Assume that $p^{k}A$, $ann(p^{k})$, $ ann(p^{2k})$ are ideals in $A$ and $(p^{k}A)*ann(p^{2k})\subseteq ann(p^{k})$. Let $\wp^{-1}:p^{k}A\rightarrow A$ be defined as in the section \ref{pullback}. Let $\odot $ be defined as 
\[[x]\odot [y]=[\wp^{-1}((p^{k}x)*y)].\] Then $\odot $ is a well defined operation on $A$.
\end{theorem}
\begin{proof} The same proof as in \cite{paper1} works.
\end{proof}

\subsection{Properties $1$ and  $2$}\label{property2}

\begin{definition}
 Let $A$ be a pseudobrace of cardinality $p^{n}$ for some prime number $p>3$ and some natural number $n$. We will say that $A$ satisfies {\em property $1$ } if 
\begin{enumerate}  
 \item for every $a, b\in A$, \[a*(a*( \cdots a*b))\in pA,\]
 \item and for every $a\in A$ and every $i>0$ 
\[a*(a*( \cdots a*ann(p^{i})))\in ann(p^{i-1})\]
\end{enumerate}
 where $a$ appears $\lfloor {\frac {p-1}4}\rfloor$ times in each of the  above expressions. 
\end{definition}

Let $A$ be a pseudobrace of cardinality $p^{n}$ for some prime number $p>3$ and some natural number $n$. We fix natural number $k>0$.

\begin{definition} Let $k$ be a natural number. We say that  a pseudobrace $A$  satisfies property $2$ (for this given number $k$)  if  for each $a,b\in A$, 
 we have that $p^{2k}a=p^{2k}b$ if and only if $p^{k}a^{\circ p^{k}}=p^{k}b^{\circ p^{k}}$. We also assume that  $p^{i}A$, $ann(p^{i})$ are ideals in $A$ for each $i$ and $(p^{i}A)*ann(p^{j})\subseteq ann(p^{j-i})$ for all $j\geq i$.
\end{definition}
 
\begin{lemma}\label{j} Let $A$ be a pseudobrace of cardinality $p^{n}$ for some prime number $p$ and some $n$.  Let $f:A\rightarrow A$ be a function  such that $p^{k}f(a)=a^{\circ p^{k}}$ for each $a\in A$.
  Then property $2$ implies that  the map $[a]_{ann (p^{2k})}\rightarrow [f(a)]_{ann (p^{2k})}$  is well defined and one-to-one.
\end{lemma}
\begin{proof}
 This follows from the fact that for $x,y\in A$, we have $[x]_{ann(p^{2k})}= [y]_{ann (p^{2k})} $ if and only if  
$x-y\in ann(p^{2k})$. This in turn is equivalent to $p^{2k}(x-y)=0$. So the map  $[a]_{ann (p^{2k})}\rightarrow [f(a)]_{ann (p^{2k})} $  is well defined and injective. Since  set $A/ann(p^{2k})$ is finite it follows that ths  is a bijection on $A/ann(p^{2k})$.
\end{proof} 
 \begin{lemma}\label{g(a)}
 Let $A$ be a pseudobrace satisfying property $2$. Let notation be as in Lemma \ref{j}. For $a\in A$
 define \[g(a)=f^{(p^{n}!-1)}(a)\] where $f^{(1)}(a)=f(a)$ and for every $i$ we denote $f^{(i+1)}(a)=f(f^{(i)}(a))$. 
 Then \[[f(g(a))]=[g(f(a))]=[a].\]
  Moreover
  \[[f(a)]=[g^{(p^{n}!-1)}(a)]\]  where $g^{(1)}(a)=a$ and for every $i$ we denote $g^{(i+1)}(a)=g(g^{(i)}(a))$.
 \end{lemma}
\begin{proof} 
  Observe first that since $[a]\rightarrow [f(a)]$ is a bijection on $A/ann(p^{2k})$, then 
  for every $i$, $[f^{(i)}(a)]\rightarrow [f^{(i+1)}(a)]$ is a bijective function on $A/ann(p^{2k})$, and hence for every $k>0$
  $[a]\rightarrow [f^{(k)}(a)]$ is a bijective function on $A/ann(p^{2k})$.  
 Fix $a\in A$. Notice that there are $p^{n}+1$ elements  \[[f^{(1)}(a)], [f^{(2)}(a)], \cdots , [f^{(p^{n}+1)}(a)],\] and since brace $A$ has cardinality $p^{n}$ it follows that \[[f^{(i)}(a)]=[f^{(j)}(a)],\]
for some $1\leq i<j\leq p^{n}+1$, hence $j-i\leq p^{n}$.
 Notice that the function
  $[f^{(i)}(a)]\rightarrow [a] $ is a bijective function on $A/ann(p^{2k})$. 
  Applying this function to both sides of equation
  $[f^{(i)}(a)]=[f^{(j)}(a)],$
  we obtain 
 $[a]=[f^{(j-i)}(a)].$
 Notice that $j-i\leq p^{n}$ implies that $j-i$ divides $p^{n}!$. Therefore 
  $[f^{(p^{n}!)}(a)]=[a].$
   This shows that $[f(g(a))]=[g(f(a))]=[a].$
 
 On the other hand, by the first part of this proof it follows that  $[a]\rightarrow [g(a)]$ is a bijective function, since   $[g(a)]=[f^{(p^{n}!-1)}(a)].$
 
 Observe now that $[g^{(p^{n}!-1)}(a)]=[f^{((p^{n}!-1)^{2})}(a)]=[f(a)],$
  since $[f^{(p^{n}!)}(a)]=[a]$.
\end{proof}

 \subsection{Theorem \ref{71}}
The aim of this section is the following result: 
\begin{theorem}\label{71}
 Let $A$ be a pseudobrace which satisfies Property $1$ and Property $2$. Let $k$ be a natural number.
 Let $\rho ^{-1}:pA\rightarrow A$ be any function such that $p\cdot \rho ^{-1}(x)=x$ for each $x\in A$. Denote $[a]=[a]_{ann(p^{2k})}$ for $a\in A$.
 Let $a,b\in A$, then $[a]*[b]$ is obtained by applying
 operations $+$, $\odot $, 
 to elements $[a]$ and $[b]$ and $\rho ^{-1}$ to elements from $pA$, and the order of applying these operation depends only on the additive group of $A$ (it does not depend on $a$ and $b$ and it does not depend on any properties of $*$).  Moreover, the result
 does not depend of the choice of function $\rho ^{-1}$. 
\end{theorem}
\begin{proof}  We will present an outline of the proof of Theorem \ref{71}. A detailed proof  is presented in the next Section \ref{next}.
 The outline of the proof:
\begin{enumerate}
\item  We introduce elements $e_{i}(a,b)$, ${\tilde e}_{i}(a,b)$, ${\bar e}(a,b)$ for $i=0,1,2, \ldots $ and related elements $e_{i}(a,b)$, ${\tilde e}_{i}(a,b)$, ${\bar e}(a,b)$. Let $a,b\in A$ then $e_{1}(a,b)'=a*b$, $e_{2}(a,b)=a*(a*b)$, etc.   
\item Next we introduce operation $\star $ such that $p^{k}(a\star b)=(a^{\circ p^{k}})*b$. Next we define ${\tilde e}_{1}'(a,b)=a\star b$, ${\tilde e}_{2}'(a,b)=a\star (a\star b)$, etc. 
\item Next we introduce operation $\odot $ such that $p^{2k}([a]\odot [b])=[p^{k}a]*[p^{k}b]$. Next we define  ${\bar e}_{1}'([a], [b])=[a]\star [b]$, ${\tilde e}_{2}'([a],[b])=[a]\star  ([a]\star [b])$, etc.
Recall that for $a\in a$ $[a]=\{a+r:r\in ann(p^{2k})\}$ denote elements of brace $A/ann(p^{2k})$.
\item We then express elements $e_{i}(a,b)$ using elements ${\tilde e}_{i}(a,b)$. Next we will express elements $[{\tilde e}_{i}(a,b)]$ using elements ${\bar e}_{i}([a],[b]).$ 
 Consequently element $e_{1}(a,b)=a*b$ can be expressed using elements $[a]\odot [b]$, which gives the conclusion.
\end{enumerate}
\end{proof}
\section{Proof of Theorem \ref{71}}\label{next}
\subsection{Elements $e_{i}(a,b)$}\label{2222}

 Let $A$ be a pseudobrace of cardinality $p^{n}$ for some prime number $p$ and for a natural number $n$. Fix a natural number $k$.

 Let $\wp ^{-1}:p^{k}A\rightarrow A$ be defined as in Section \ref{pullback}, so $\wp^{-1}:p^{k}A\rightarrow A$ is such that $p^{k}\wp^{-1}(x)=x$ for every $x\in p^{k}A$.
 We will assume that 
\[\wp^{-1}(a)=(\rho ^{-1} )^{k}(a)=\rho^{-1}( \cdots (\rho^{-1}(a)\cdots )),\] where $\rho ^{-1}:pA\rightarrow A$ is a given ``pullback'' function $\rho ^{-1}:pA\rightarrow A$ such that $p\rho^{-1}x=x$, for $x\in pA$.
 Observe that
\[\rho ^{-1}(a+b)-\rho^{-1}(a)-\rho ^{-1}(b)\in ann(p),\]
for all $a, b\in A$.

We start with the definition of elements $e_{i}'(a,b)$, $e_{i}(a,b)$ for $a,b\in A$. 
\begin{definition}\label{defi}
 Let $A$ be a pseudobrace of cardinality $p^{n}$ for some prime number $p>3$ and for a natural number $n$. Assume that $A$ satisfies Property $1$.
Let $a,b\in A$. Denote $e_{0}'(a,b)=b$.
  Define $e_{1}'(a,b)=a*b$ and inductively define \[e_{i+1}'(a,b)=a*e_{i}'(a,b).\]

Define \[e_{j}(a, b)= e_{j}'(a, b)\] for $j$ not divisible by $p$. 

Define  \[e_{0}(a,b)=e_{\frac {p-1}2}(a, \rho^{-1}(b))\] for $b\in pA$.
 
Let $j$ be divisible by $p$ and let $s$ be maximal number such that $p^{s}$ divides $j$.  For $j\geq p$ define
\[e_{j}(a, b)=d_{s}(a, e_{j-{\frac{(p-1)s}2}}'(a,b))\] where $d_{1}(a,b)=e_{0}(a,b)$ and inductively it is defined as $d_{i+1}(a,b)=e_{0}(a, d_{i}(a,b))$.
 Notice that $j-{\frac{(p-1)s}2}\geq {\frac {(p+1)s}2}$, therefore $e_{j-{\frac{(p-1)s}2}}'(a,b)\in pA$ (by the property $1$). 
Consequently if $j>0$ then $e_{j}'(a,b)$ is well defined for all $a,b\in A$.

 Let $\upharpoonleft $ be a symbol, we denote $e_{1}'(a, 1)=a$ and $e_{i+1}'(a,\upharpoonleft )=a*e_{i}'(a, \upharpoonleft )$. We then define elements $e_{j}(a,\upharpoonleft )$ as above.
Denote $e_{j}(a)=e_{j}(a, \upharpoonleft )$.  
\end{definition}

\begin{lemma}\label{e_{0}}  Let $A$ be a pseudobrace of cardinality $p^{n}$ for some prime number $p>5$ and for a natural number $n$. Assume that $A$ satisfies Property $1$, then 

 \[e_{\frac{p-3}2}(a, A)\subseteq pA, \] 
and \[ e_{\frac{p-3}2}(a, ann(p))=0.\]
 Moreover 
 \[ e_{0}(a,b)= e_{\frac {p-3}2}(a, \rho^{-1}(a*b)).\]
 Moreover, $e_{\frac{p-3}2}(a, \upharpoonleft )\subseteq pA$.
\end{lemma}
\begin{proof} It immediately follows from property $1$ since ${\frac{p-3}2}< \lfloor {\frac {p-1}4}\rfloor$ and since $a\rho^{-1}(b)-\rho ^{-1}(a*b)\in ann(p)$.
\end{proof}

$ $ 

\begin{remark}\label{a*a} Let $A$ be a pseudobrace of cardinality $p^{n}$ for some prime number $p>5$ and for a natural number $n$.  Assume that  $A$ satisfies  Property $1$. 
Notice that $e_{0}(a,b)$  does not depend of the choice of function 
$\rho ^{-1}$.
 Indeed, by the definition , if $\rho '^{-1}$  and $\rho ''^{-1}$ are two functions satisfting properties of the function $\rho ^{-1}$ then for
 $b\in pA$ we have $\rho '^{-1}(b)-\rho ''^{-1}(b)\in ann(p)$. The result now follows from Lemma \ref{e_{0}}.

Observe also that  \[ e_{0}(a,b)+e_{0}(a,c)=e_{0}(a, b+c),\]
 for all $a\in A$ and all $b,c\in pA$, since $\rho ^{-1}(b+c)-\rho^{-1}(b)-\rho ^{-1}(c)\in ann(p)$. 
 Consequently, for each $i>0$ we have
\[e_{i}(a, b+c)=e_{i}(a,b)+e_{i}(a,c)\] for all $a\in A$, $b,c\in A\cup \{ \upharpoonleft \}$. 
\end{remark}
\begin{definition}   Let $A$ be a pseudobrace of cardinality $p^{n}$ for some prime number $p>5$ and for a natural number $n$.  Assume that  $A$ satisfies  Property $1$. Let $a\in A$, $b\in A\cup \{ \upharpoonleft \}$. 
 Define ${ C}_{i}^{t}(a,b)\subseteq A$ as follows:
\begin{itemize}
\item 
 $C_{i}^{0}(a,b)$ consists of element ${ e}_{i}'(a,b)$ for $i=1,2, \ldots .$
\end{itemize}

 Let $j>0$ then sets $C_{i}^{j}(a,b)$ are only defined for $i\geq {\frac {p-3}2}$. We define $C_{i}^{j}(a,b)$ be the smallest subsets of $A$ satysfying:
\begin{itemize}
\item  If $i\geq {\frac {p-3}2}$ then ${e}_{0}(a ,C_{i}^{j}(a,b))\in C_{i}^{j+1}(a,b)$.
 \item If $l\geq 0$ then  ${ e}_{l}'(a,C_{i}^{j}(a,b))\in C_{i+l}^{j}(a,b)$.
\end{itemize}
\end{definition}

 \begin{lemma}\label{wazny}
 Let $A$ be a pseudobrace of cardinality $p^{n}$ for some prime number $p>5$ and for a natural number $n$.  Assume that  $A$ satisfies  Property $1$. Then the following holds:
\begin{enumerate}
\item $C_{t}^{k}(a,b)\subseteq e_{k+t}'(a,A).$
\item  For $j$ divisible by $p$ let $s_{j}$ be the largest integer such that $p^{s_{j}}$ divides $j$. Then $p^{k-s_{j}}$ divides 
${{p^{k}}\choose j}$ since ${{p^{k}}\choose j}={{p^{k}-1}\choose {j-1}}
\cdot {\frac {p^{k}}j}.$  Let $\sigma _{j}$ be integers such that  ${{p^{k}}\choose j}=p^{k-s_{j}}\sigma _{j}.$ 
  Then $p^{s_{j}}e_{j}(a,b)=e_{j}'(a,b)$, and $p^{s_{j}}e_{j}(a)=e_{j}'(a)$.

 Moreover $ a^{\circ p^{k}}*b=p^{k}\sum_{j=1}^{p^{k}}\sigma _{j}e_{j}(a,b)$ and 
 \[a^{\circ p^{k}}=p^{k}\sum_{j=1}^{p^{k}}\sigma _{j}e_{j}(a).\]
\end{enumerate}
 \end{lemma}
\begin{proof} (1) By induction on $t$ (and for a given $t$ we will use induction on $k$). If $t<{\frac {p-3}2}$ then $k=0$ (as otherwise $C_{t}^{k}(a,b)$ is not defined).
 Note that $C_{t}^{0}(a,b)=\{e_{t}'(a,b)\}$ for any $t\geq 1$, so the result holds.
 If $k>0$ and $c\in C_{t}^{k}(a,b)$ then either  $c\in e_{l}'(a, C_{t-l}^{k}(a,b))$ for some $l>0$ or $c\in e_{0}(a, C_{t}^{k-1}(a,b))$ for some $t\geq {\frac {p-3}2}$.
 By the inductive assumption \[e_{l}'(a, C_{t-l}^{k}(a,b))\subseteq e_{l}'(a, e_{t-l+k}'(a, A))=e_{t+k}'(a,A)\] as required.
 
 By the inductive assumption \[e_{0}(a, C_{t}^{k-1}(a,b))\subseteq e_{\frac {p-1}2}(a, \rho ^{-1}(e_{t+k-1}'(a,A)))=e_{\frac {p-1}2}(a, e_{t+k-{\frac {p-1}2}}'(a, \rho ^{-1}(e_{\frac {p-3}2}'(a, A))))\] 
 where the last inclusion follows from the fact that 
\[\rho ^{-1}(e_{t+k-1}'(a,c))-e_{t+k-{\frac {p-1}2}}'(a, \rho ^{-1}(e_{\frac {p-3}2}'(a,b))\in ann(p)\] and 
 $e_{\frac {p-1}2}(a, ann(p))= 0$ (by Lemma \ref{e_{0}}).
  Consequently $e_{0}(a, C_{t}^{k-1}(a,b))\subseteq e_{t+k}'(a, A)$.
$ $

(2)
 Observe that  
 $p^{s_{j}}e_{j}(a,b)=e_{j}'(a,b)$ implies $p^{k}\sigma_{j}e_{j}(a,b)={{p^{k}}\choose j}e_{j}'(a,b) .$
 
 Recall that in any brace and pseudobrace $a^{\circ p^{k}}=a\circ (a\circ( \cdots a\circ a)))$, and so
$a^{\circ p^{k}}=\sum_{j=1}^{p^{k}}{{p^{k}}\choose j} e_{j}'(a)  =p^{k}\sum_{j=1}^{p^{k}}\sigma _{j}e_{j}(a)$.
  Moreover, by the definition in any pseudobrace we have
$a^{\circ p^{k}}\circ b=a\circ (a\circ( \cdots a\circ b)))$ and this implies  $ a^{\circ p^{k}}*b=p^{k}\sum_{j=1}^{p^{k}}\sigma _{j}e_{j}(a,b)$. 
 \end{proof}
\subsection{ Elements ${\tilde e}_{i}(a,b)$}\label{3333}
In this section we introduce function $f:A\rightarrow A$ and 
operation $\star $.

\begin{definition}\label{diam} Let notation be as in Lemma \ref{wazny}. Let $a, b\in A$.  We define 
\[f(a)=\sum_{j=1}^{p^{k}}\sigma _{j}e_{j}(a).\]

 We also denote: 
\[a\star b=\sum_{j=1}^{p^{k}}\sigma _{j}e_{j}(a,b).\]

Note that 
$a \star (b+c)=a\star b +a \star c $ (it follows from the Remark \ref{a*a}).
\end{definition}

 We will now define analogons of elements $e_{i}'(a,b)$ and $e_{i}(a,b)$ but we will use operation $\star $ instead of operation $* $ to define them.

\begin{definition} Let assumptions be as in Lemma \ref{wazny}.
Define ${\tilde e}'_{1}(a,b)=a\star b$, and inductively  ${\tilde e}'_{i+1}(a,b)=a\star {\tilde e}'_{i}(a,b)$
 and similarly, ${\tilde e}'_{1}(a)=a$, and inductively  ${\tilde e}'_{i+1}(a)=a\star {\tilde e}'_{i}(a)$.
\end{definition}

\begin{remark}\label{3} Let assumptions be as in Lemma \ref{wazny} then  $a\star A\subseteq a*A$ by definition of $\star $.
 Moreover, by Lemma \ref{wazny} (1) and definition of $\star $ we have 
\[a\star e_{i}'(a,A)\subseteq e_{i+1}'(a, A)\] for each $i>0$. Hence, 
\[{\tilde e}_{j}'(a, e_{i}'(a,A))\subseteq e_{i+j}'(a, A)\] for each $i,j>0$.
 Consequently ${\tilde e}'_{i}(a,b)\in e'_{i}(a,A)$ for all $i\geq 1$. Since $e_{\frac{p-3}2}(a,A)\in pA$ then 
 \[{\tilde e}_{\frac {p-3}2}(a,b)\in pA.\]
 Observe also that \[{\tilde e}'_{i}(a,{\tilde e}'_{j}(a,b))={\tilde e}'_{i+j}(a,b)\] for $i,j>0$.

Note that  by the last part of Remark \ref{a*a}
\[{\tilde e}'_{i}(a,b+c)={\tilde e}'_{i}(a,b)+{\tilde e}'_{i}(a,c),\]
 for all $i>0$ and all $a,b,c\in A$.
 \end{remark}

Next we define   elements ${\tilde e}_{j}(a,b)$ and elements ${\tilde e}_{j}(a)$. 

\begin{definition}\label{xx} Let assumptions be as in Lemma \ref{wazny}. 
Let $a, b\in A$. Define 
\[{\tilde e}_{j}(a, b)= {\tilde e}_{j}'(a, b)\]  for $j$ not divisible by $p$.

For $b\in pA$ define  \[{\tilde e}_{0}(a,b)={\tilde e}_{\frac {p-1}2}(a, \rho^{-1}(b)).\]

Let $j\geq p$ be divisible by $p$ and let $s$ be maximal number such that $p^{s}$ divides $j$, the  define
\[{\tilde e}_{j}(a, b)={\tilde d}_{s}(a,  {\tilde e}_{j-{\frac{(p-1)s}2}}'(a,b) )\] where ${\tilde d}_{1}(a,b)={\tilde e}_{0}(a,b)$ and inductively it is defined as ${\tilde d}_{j+1}(a,b)={\tilde e}_{0}(a, {\tilde d}_{j}(a,b))$.
 Notice that $j-{\frac{(p-1)s}2}\geq {\frac {(p+1)s}2}$, therefore ${\tilde e}_{j-{\frac{(p-1)s}2}}(a,b)\in pA$. 
Consequently if $j>0$ then ${\tilde e}_{j}(a,b)$ is well defined for all $a,b\in A$.

 Let $\upharpoonleft $ be a symbol. We denote ${\tilde e}_{1}'(a, 1)=a$ and ${\tilde e}'_{i+1}(a,\upharpoonleft )=a \star {\tilde e}'_{i}(a, \upharpoonleft )$. We then define elements ${\tilde e}_{j}(a,\upharpoonleft )$ as above.
Denote ${\tilde e}_{j}(a)={\tilde e}_{j}(a, \upharpoonleft )$.  
\end{definition}

\begin{remark}\label{yyy} Let assumptions be as in Lemma \ref{wazny}. By Remark \ref{3} and then by Lemma 
\ref{e_{0}}    we have

\[{\tilde e}_{\frac {p-3}2}(a, A)\subseteq {e}_{\frac {p-3}2}(a,ann(p))=0.\]

 Observe that this implies that  ${\tilde e}_{0}(a,b)$ does not depend of the choice of function $\rho ^{-1}$ and 
\[{\tilde e}_{0}(a,b+c)={\tilde e}_{0}(a,b)+{\tilde e}_{0}(a,c),\]
for all $a\in A$, $b,c\in pA$. 
\end{remark}
\begin{lemma}\label{2}
 Let assumptions be as in Lemma \ref{wazny}.  Let $a\in A, b\in pA$. If $i\geq {\frac {p-3}2}$ then  
  \[{\tilde e}_{0}(a,b)={\tilde e}_{\frac {p-3}2}(a, \rho^{-1}(a\star b)).\]
 Moreover  
\[{\tilde e}_{0}(a, e_{i}'(a, A))\subseteq e_{i+1}'(a,A).\] 

\end{lemma} 
\begin{proof} 
 Observe that 
  \[{\tilde e}_{0}(a,b)={\tilde e}_{\frac {p-3}2}(a, \rho^{-1}(a\star b))\] follows from the fact that 
  ${\tilde e}_{\frac {p-3}2}(a, ann(p))=0$ by Remark \ref{yyy}, and because 
\[\rho ^{-1}(a\star b)-a\star \rho ^{-1}(b)\in ann(p).\]
 Since $a\star pc=p(a\star c)$ by  Remark \ref{a*a}.

$ $
 
Let $i\geq {\frac {p-3}2}$. 
 By applying the fact that \[{\tilde e}_{0}(a,b)=a\star {\tilde e}_{\frac {p-3}2}(a, \rho^{-1}(b))={\tilde e}_{\frac {p-3}2}(a, \rho^{-1}(a\star b))\]
 several times we obtain 
 \[{\tilde e}_{0}(a, {\tilde e}_{i}'(a, b))={\tilde e}_{i+1}'(a, \rho ^{-1}({\tilde e}_{\frac {p-3}2}(a, b))).\] 
\end{proof}

\begin{definition}  Let assumptions be as in Lemma \ref{wazny}. We define set ${\tilde C}_{i}^{t}(a,b)\subseteq A$ as follows:
\begin{itemize}
\item  ${\tilde C}_{i}^{0}(a,b)$ consists of element ${\tilde e}_{i}'(a,b)$ for $i\geq 1$.
\end{itemize}

  Let $j>0$ then sets ${\tilde C}_{i}^{j}(a,b)$ are only defined for $i\geq {\frac {p-3}2}$. We define ${\tilde C}_{i}^{j}$ be the smallest subsets of $A$ satisfying:
\begin{itemize}
\item  If $a\in A$, $b\in {\tilde C}_{i}^{j-1}(a,b)$ then ${\tilde e}_{0}(a,b)\in {\tilde C}_{i}^{j}(a,b)$, for $i\geq {\frac {p-3}2}$.
 \item If $a\in A$, $b\in {\tilde C}_{i}^{j}(a,b)$ then ${\tilde e}_{l}'(a,b)\in {\tilde C}_{i+l}^{j}(a,b)$, for $l>0$.
\end{itemize}
\end{definition}

\begin{remark} Observe that by  Remark \ref{3} we have
${\tilde e}_{j}'(a, e_{i}'(a,A))\subseteq e_{i+j}'(a, A)$.
 By Lemma \ref{wazny} (1)  for $i\geq {\frac {p-3}2}$ we have 
\[{\tilde e}_{0}(a, e_{i}'(a,c))={\tilde e}_{\frac {p-1}2}(a, e_{i-{\frac {p-3}2}}'(a, \rho ^{-1}(e_{\frac {p-3}2}(a,c) )))   
\in  e_{i+1}'(a, A).\]  Consequently ${\tilde C}_{t}^{j}(a,b)\subseteq e_{j+t}'(a, A).$
\end{remark}

\subsection{ Connections between operations  $\star $ and $*$ }\label{31}

 In this section we will investigate some properties of the binary operation $\star$.
  We start with a definition:

\begin{definition} 
Let $a,b\in A$ and $i,t\geq 0$ be natural numbers. Define for $t>0$ 
 \[E_{i}^{t}(a,b)=\sum_{ i'\geq {\frac {p-3}2}, i'\geq i, t'\geq t} C_{i'}^{t'}(a,b).\]
 This means that $E_{i}^{t}(a,b)$ is the additive subgroup of $(A, +)$ generated by elements from sets $C_{i'}^{t'}(a,b)$  with $i'\geq {\frac {p-3}2}, i'\geq i, t'\geq t$.
And for $t=0$ and $i\geq 1$ define 
 \[E_{i}^{0}(a,b)=\sum_{i'\geq i, t'\geq 0} C_{i'}^{t'}(a,b).\]
 This means that $E_{i}^{0}(a,b)$ is the additive subgroup of $(A, +)$ generated by elements from sets $C_{i'}^{t'}(a,b)$  with $i'\geq i, t'\geq 0$.
 \end{definition}

 Let $A$ be a brace and $S\subseteq A$, $a\in A$ then we denote
\[a+S=\{a+s:s\in S\}.\]
 Our first lemma in this section is the following:
\begin{lemma}\label{prel7} Let assumptions be as in Lemma \ref{wazny}. 
 Let $i$ be a natural number. Then the following holds: 
 \[{\tilde e}'_{i}(a,b)\in e_{i}'(a,b)+ w_{i}(a,b)\]
where $w_{i}(a,b)\in  E_{i+1}^{0}(a,b)+E_{i}^{1}(a,b).$ Therefore $w_{i}(a,b)$ is obtained by applying operations $e_{j}(a, .)$ for $j\geq 0$ to element $b$ several times, and adding such elements.
The order of applying these operations  in the construction of the element $w_{i}(a,b)$ 
does  not depend on elements $a,b$ and it does not depend on properties of $*$, it only depends on the cardinality of $A$.
 Note that $e_{i}'(a,b)\in E_{i}^{0}(a,b)$. 
\end{lemma}
\begin{proof} We proceed by induction on $i$. For $i=1$ the result follows from Lemma \ref{wazny} since ${\tilde e}_{1}(a,b)=a\star b$.
 
Suppose that the result holds for some $i\geq 1$ then for $i+1$ we have 
\[{\tilde e}'_{i+1}(a,b)=a\star  {\tilde e}'_{i}(a,b)=\sum_{j=1}^{p^{k}}\sigma _{j}e_{j}(a,{\tilde e}'_{i}(a,b)).\]
 By the inductive assumption ${\tilde e}'_{i}(a,b)\in e_{i}'(a,b)+E_{i+1}^{0}(a,b) +E_{i}^{1}(a,b)$.
Therefore \[e_{j}(a,{\tilde e}'_{i}(a,b))\in e_{j}(a, {e}_{i}'(a,b))+ e_{j}(a, E_{i+1}^{0}(a,b))+e_{j}(a, E_{i}^{1}(a,b)).\]

$ $
 We proceed by induction on (j) (for this fixed $i$).

{\em Case $1$.} 
For $j=1$ we have $e_{j}=e_{1}$ so $e_{1}(a, {e}_{i}'(a,b))=a*{e}_{i}'(a,b)=e_{i+1}'(a,b)$ and 
  $e_{1}(a, E_{i+1}^{0}(a,b))=a*E_{i+1}^{0}(a,b)\subseteq E_{i+2}^{0}(a, b).$
 Moreover, $e_{1}(a, E_{i}^{1}(a,b))\subseteq  E_{i+1}^{1}(a, b).$

$ $

{\em Case $2$.} 
If $1<j<p$, or more generally if $j$ is not divisible by $p$,  by using similar reasoning, we get  $e_{j}=e_{j}'$, hence  $e_{j}(a, {e}_{i}'(a,b))=e_{j}'(a, {e}_{i}'(a,b))=e_{i+j}'(a,b)\in 
E_{i+2}^{0}(a,b)$ and $e_{j}(a, E_{i}^{1}(a,b))=e_{j}'(a,E_{i}^{1})(a,b)\subseteq E_{i+j}^{1}(a,b)\subseteq E_{i+2}^{1}(a,b)$.

$ $

{\em Case $3$.} Suppose now that $j>p-1$ and $j$ is divisible by $p$.
 Let $s$ be the largest integer such that $p^{s}$ divides $j$ then by Lemma \ref{wazny},
  $e_{j}(a, {e}_{i}'(a,b))=d_{s}(a, e_{j-{\frac {(p-1)s}2 }}'(a, e_{i}'(a,b)))$.

 We calculate  \[e_{j-{\frac {(p-1)s}2 }}'(a,e_{i}'(a,b))= e_{j-{\frac {(p-1)s}2 }+i}'(a,b)\subseteq E_{i+{\frac {p-1}2}}^{0}(a,b)\] since $j-{\frac {(p-1)s}2 }\geq{\frac {p+1}2}$.   Consequently, $e_{j}(a, {e}_{i}'(a,b))\subseteq d_{s}(a,E_{i+{\frac {p-1}2}}^{0}(a,b))\subseteq E_{i+1}^{1}(a,b)$.

Note that $e_{i}'(a,b), E_{i}^{1}(a,b), E_{i+1}^{0}(a,b)\subseteq E_{i}^{0}(a,b)$. Observe that  
 \[e_{j-{\frac {(p-1)s}2 }}'(a,E_{i}^{0}(a,b)) \subseteq E_{i+j-{\frac {(p-1)s}2 }}^{0}(a,b)\subseteq E_{i+{\frac {p+1}2}}^{0}(a,b).\] 

Therefore \[e_{j}(a, E_{i}^{0}(a,b))\subseteq d_{s}(a, E_{i+{\frac {p+1}2}}^{0}(a,b)) \subseteq E_{i+1}^{1}(a,b).\]
 
This concludes the proof. 
\end{proof}

\begin{lemma}\label{7useful} Let notation be as in Lemma \ref{prel7}. We assume $p> 3$. Let $i,t$ be natural numbers such that $t\geq 0$ and $i\geq {\frac {p-3}2}$.
 Let $a,b\in A$,  $c\in  C_{i}^{t}(a,b)$, then 
\[{\tilde e}_{0}(a,c)\in e_{0}(a, c)+w_{0}(a,b)\]
where $w_{0}(a,b)\in E_{i}^{t+2}(a,b)+E_{i+1}^{t+1}(a,b).$ Therefore $w_{0}(a,b)$ is obtained by applying operations $e_{j}(a, .)$ for $j\geq 0$ to element $b$ several times, and adding such elements.
The order of applying these operations  in the construction of the element $w_{0}(a,b)$ 
does  not depend on elements $a,b$ and it does not depend on properties of $*$, it only depends on the cardinality of $A$.
Note that $e_{0}(a,c)\in E_{i}^{t+1}(a,b)$.
\end{lemma}
\begin{proof} Recall that \[{\tilde e}_{0}(a,c)={\tilde e}_{\frac {p-1}2}'(a, \rho ^{-1}(c)).\]
  We can apply Lemma \ref{wazny} to $d=\rho ^{-1}(c)\in \rho ^{-1} (C_{i}^{t}(a,b))$. Let $t>0$ then $i\geq {\frac {p-1}2}$ by assumptions.
Then we obtain
\[a\star\rho ^{-1}(c)=\sum_{j=1}^{p^{k}}\sigma _{j}e_{j}(a,\rho ^{-1}(c)).\]
 
Therefore \[{\tilde e}_{0}(a,c)={\tilde e}_{\frac{p-1}2}(a,\rho ^{-1}(c))={\tilde e}_{\frac{p-3}2}(a,a\star \rho ^{-1}(c))=\sum_{j=1}^{p^{k}}\sigma _{j}{\tilde e}_{\frac {p-3}2}'(a, e_{j}(a,\rho ^{-1}(c))).\]

{\em Case 1.} Let $j> p-1$ and $j$ not divisible by $p$, then \[e_{j}((a,\rho ^{-1}(c))=e_{j-{\frac {p-1}2}}'(a, e_{\frac{p-1}2}'(a,\rho^{-1}(c))=
e_{j-{\frac {p-1}2}}'(a, e_{0}(a,c))\subseteq E_{i+1}^{t+1}(a,b).\]

It follows by Lemma \ref{prel7} that for $j>{\frac {p-1}2}$ and $j$ not divisible by $p$ we have 
\[{\tilde e}_{\frac {p-1}2}(a, e_{j}(a, \rho ^{-1}(c))) \subseteq E_{i+1}^{t+1}(a,b). \]

$ $

{\em Case 2.}  Let $j$ be divisible by $p$ and $s$ be the largest number such that $p^{s}$ divides $j$,  then $j-{\frac {(p-1)s}2}\geq {\frac {(p-1)}2}+1$
 and consequently for some $i'\geq 1$ we have 
\[e_{j}(a, \rho^{-1}(c))=  d_{s}(a, e_{{\frac {(p-1)}2}+i'}'(a,\rho ^{-1}(c)))= d_{s}(a, e_{i'}'(a, e_{0}(a,c)))\in C_{i+i'}^{t+s+1 }(a,b)\subseteq E_{i+1}^{t+1}(a,b).\]   

Then we obtain: 
\[{\tilde e}_{\frac {p-3}2}(a, e_{j}(a, \rho ^{-1}(c))) \subseteq E_{i+1}^{t+1}(a,b). \]

$ $
{\em Case 3.} It remains to show that if  $1\leq j\leq   {p-1}$ then
  
\[{\tilde e}_{\frac {p-3}2}(a, e_{j}(a, \rho ^{-1}(c)))-e_{0}(a,c)\in  E_{i+1}^{t+1}(a,b)+E_{i}^{t+2}(a,b). \]

 By Lemma \ref{2}, for  $1\leq j\leq   {p-1}$  we have 
\[e_{p}(a, e_{j}(a, \rho ^{-1}(c)))=e_{j+1}(e_{0}(a, e_{0}(a,c)))\subseteq E_{i+j+1}^{t+2}(a, b)\subseteq E_{i+1}^{t+2}(a,b).\] 
 Moreover for $l>p$ and  divisible by $p$ similarly we have $e_{l}(a, e_{j}(a, \rho ^{-1}(c)))\subseteq E_{i+1}^{t+2}(a, b).$
And for $l\geq p-1$ and not divisible by $p$ we have 
\[e_{l}'(a, e_{j}(a, \rho ^{-1}(c))=e_{l+j}'(a, \rho ^{-1}(c))=e_{l+j-{\frac {p-1}2}}'(a, e_{0}(a,c))\subseteq E_{i+1}^{t+1}(a,b).\]

 Observe that for $0<l, j\leq p-1$ we have
\[e_{l}(a, e_{j}(a,  \rho ^{-1}(c)))=e_{j+l}'(a, \rho ^{-1}(c)).\]

 Continuing in this way we get for $1< j\leq p-1$
 \[{\tilde e}_{\frac {p-3}2}(a, e_{j}(a, \rho ^{-1}(c)))\in   E_{i+1}^{t+1}(a,b)+E_{i}^{t+2}(a,b). \]
 And for $j=1$ we get 
 \[{\tilde e}_{\frac {p-3}2}(a, e_{j}(a, \rho ^{-1}(c)))\in e_{0}(a,c)+  E_{i+1}^{t+1}(a,b)+E_{i}^{t+2}(a,b). \]
 
 $ $
 
 Conclusion: By Cases $1, 2, 3, 4$ we get
 \[{\tilde e}_{0}(a,c)={\tilde e}_{\frac {p-3}2}(a, a\star \rho ^{-1}(c))\subseteq e_{0}(a,c)+ E_{i+1}^{t+1}(a,b)+E_{i}^{t+2}(a,b).\] 
\end{proof}

\begin{definition} Let $c(a,b)\in {C}_{i}^{t}(a,b)$. We say that $c'(a,b)$ is the corresponding element to $c(a,b)$ when 
  ${\tilde c}(a,b)$ is obtained in the same way as $c(a,b)$ but using elements ${\tilde e}_{i}(a,b)$ in the construction instead of elements $e_{i}(a,b)$
 (in other words the only difference is that it uses operation $\star $ instead of operation $*$). Notice that ${\tilde c}(a,b)\in {\tilde C}_{i}^{t}(a,b).$
\end{definition}

 In the following results  the meaning of ''the presentation only depends on the cardinality of $A$'' is the same as in Lemma \ref{7useful}.  Observe that the formula for the construction does not depend on the additive group of  $A$ (as it depends only on the cardinality of $A$) but the result will depend on the additive group. Usually the result will be different  for braces with different additive groups (similarly as in the construction of the group of flows, which does not depend on the additive group, but the result depends on this group).  

\begin{lemma}\label{new7} Let notation be as in Lemma \ref{wazny}. Let $t\geq 0,i>0$ be natural numbers such that if $t>0$ then $i\geq {\frac {p-3}2}$. Let $p>3$ be a prime number.  
 Let $c(a,b)\in {C}_{i}^{t}(a,b)$ and ${\tilde c}(a,b)\in {\tilde C}_{i}^{t}(a,b)$  be the corresponding element to $c(a,b)$. Then,  
\[{\tilde c}(a,b)\in c(a,b)+E_{i+1}^{t}(a,b) +E_{i}^{t+1}(a,b)\] and this presentation
 only depends  on the cardinality of pseudobrace $A$ 
(and it does not depend on elements $a$ and $b$ as formulas
 from Lemmas \ref{prel7} and \ref{7useful} do not depend on elements $a,b$). 
\end{lemma}
\begin{proof} Note that sets ${\tilde C}_{i}^{t}(a,b)$ are obtained by applying operations ${\tilde e}_{j}'(a,d)$ for $d\in {\tilde C}_{i-j}^{t}(a,b)$ and also by applying operation  $\tilde e_{0}(a,d)$ for  ${\tilde C}_{i}^{t-1}(a,b)$. The result now follows by induction from Lemmas 
 \ref{prel7} and \ref{7useful}. 
\end{proof}
\begin{corollary}\label{tilde}
 Let notation be as in Lemma \ref{new7} then ${\tilde e}_{\frac {p-1}2}(a,b)\in pA$ for every $a,b\in A$ (by Remark \ref{3}).
\end{corollary}

\begin{definition}
Let $a,b\in A$ and $i,t\geq 0$ be natural numbers. Define for $t>0$ 
 \[{\tilde E}_{i}^{t}(a,b)=\sum_{ i'\geq {\frac {p-3}2}, i'\geq i, t'\geq t} {\tilde C}_{i'}^{t'}(a,b).\]
 This means that ${\tilde E}_{i}^{t}(a,b)$ is the additive subgroup of $(A, +)$ generated by elements from sets ${\tilde C}_{i'}^{t'}(a,b)$  with $i'\geq {\frac {p-3}2}, i'\geq i, t'\geq t$.
And for $t=0$ and $i\geq 1$ define 
 \[{\tilde E}_{i}^{0}(a,b)=\sum_{i'\geq i, t'\geq 0} {\tilde C}_{i'}^{t'}(a,b).\]
 This means that ${\tilde E}_{i}^{0}(a,b)$ is the additive subgroup of $(A, +)$ generated by elements from sets ${\tilde C}_{i'}^{t'}(a,b)$  with $i'\geq i, t'\geq 0$ (only summing ${\tilde C}_{i'}^{t'}(a,b)$ which are defined, so if $i'>0$ then $t'\geq {\frac {p-3}2}$) .
 \end{definition}

\begin{lemma}\label{super7} Let notation be as in Lemma \ref{new7}, so $c(a,b)\in C_{i}^{t}(a,b)$. Then 
$c(a,b)\in {\tilde c}(a,b)+{\tilde E}_{i+1}^{t}(a,b)+{\tilde E}_{i}^{t+1}(a,b)$
 and this presentation only depends on the cardinality of of brace $A$ (it does not depend on elements $a$ and $b$ as formulas from Lemmas \ref{prel7} and \ref{7useful} do not depend of elements $a,b$). 
\end{lemma}
\begin{proof} We first introduce partial ordering on pairs of natural numbers. We say that  $(i,t)<(i',t')$ if $i\leq i'$, $t\leq i'$ and $i+t<i'+t'$ and that element $c(a,b)\in C_{i}^{t}(a,b)$ has order $(i,t)$.
 By Lemma \ref{new7}, we have  
\[{\tilde c}(a,b)\in c(a,b)+h(a,b),\] where $h(a,b)\in E_{i+1}^{t}(a,b) +E_{i}^{t+1}(a,b)$.
 Notice that $h(a,b)$ is a sum of elements with orders larger than $(i,t)$ (note that $c(a,b)$ has order $(i,t)$).
Hence we have $c(a,b)={\tilde c}(a,b)-h(a,b)$. Let $\tilde {h}(a,b)$ be the element corresponding to $h(a,b)$. By Lemma \ref{new7} we get that 
\[h(a,b)={\tilde h}(a,b)+g(a,b)\]
 where $g(a,b)$ is a sum of elements from $E_{i'}^{t'}(a,b)$ with each $(i',j')$ larger than  with orders larger than an order of a summand of $h(a,b)$. We obtain 
$c(a,b)={\tilde c}(a,b)-{\tilde h}(a,b)-g(a,b).$
Continuing in this way we increase  orders of the summands so this process will stop (since $C_{i}^{t}\in A^{i+t}$ as showed before and $A^{n}=0$).
 So after repeating such substitutions several times we will obtain the desired  conclusion.
\end{proof}

\begin{lemma}\label{6*} Let assumptions be as in Lemma \ref{super7}. Suppose that pseudobrace $A$ satisfies property $2$.
Then 
\[a\in f(a)+{\tilde E}_{1}^{0}(a, f(a)).\]
Moreover, this presentation only depends on the cardinality of of brace $A$ (it does not depend on element $a$). 
 \end{lemma}
\begin{proof} Observe that $a=f(a)+c(a,a)$ for some $c(a,a)\in E_{1}^{0}(a,a)$, this follows from the definition of $f(a)$.
 Therefore, $a-f(a)=c(a,a)$.    
By substituting for the second variable $a$ in $c(a,a)$   the above expression on the right hand side 
 we get $c(a,a)=c(a, f(a))+c(a, a-f(a))=c(a, f(a))+c(a, c(a,a))$. Notice that $c(a, c(a,a))\in c(a, E_{1}^{0}(a,a))\subseteq E_{2}^{0}(a,a)$. Consequently, 
\[  a\in f(a)+  c(a, f(a))+ E_{2}^{0}(a, a).\]
 Therefore $a-f(a)-c(a, f(a))=c'(a,a)$ for some $c'(a,a)\in E_{2}^{0}(a,a)$. By applying the same reasoning to $c'(a,a)$ instead of $c(a,a)$ we obtain \[c'(a,a)=c'(a, f(a))+c'(a, a-f(a))=c'(a, f(a))+c'(a, c(a,a)).\] Consequently 
\[  a\in f(a)+  c(a, f(a))+ c'(a, f(a))+ E_{3}^{0}(a, a).\]

 Continuing to substitute in this way 
we obtain 
\[  a\in f(a)+E_{1}^{0}(a, f(a)),\]
 (since $C_{i}^{t}(a,a)\subseteq  A^{n+t}=0$ for $i+t>n$.) 
Now we can apply 
Theorem \ref{super7} for $b=f(a)$, this gives the desired conclusion. 
\end{proof}

\section{Relations between $*$ and $\odot $}

In this section we denote $[a]_{ann(p^{2k})}=[a]$ for $a\in A$.
In this section we introduce elements ${\bar e}([a],[b])$.

\begin{definition}\label{bar} Let assumptions be as in Lemma \ref{super7}.
Define ${\bar e}_{1}'([a],[b])=[a]\odot [b]$ and inductively  ${\bar e}_{i+1}'([a],[b])=[a]\odot {\bar e}_{i}'([a],[b])$. 

\end{definition}
 Recall that by the definition of a pseudobrace as in any pseudobrace $a^{\circ p^{k}}\circ b=a\circ (a\circ \cdots (a\circ b)))$ where $a$ appears $p^{k}$ times ion the right hand side. 
Observe that $p^{k}(a\star b)=a^{\circ p^{k}}* b$, and so 
\[[f(a)]\odot [b]=[a\star b].\]

\begin{remark}\label{bar2} Let assumptions be as in Lemma \ref{super7}. Suppose that pseudobrace $A$ satisfies property $2$. 
 Observe that \[{\bar e}_{i}'([f(a)],[b])=[{\tilde e}_{i}'(a,b)].\]
 It follows from Lemma \ref{super7} and from the fact that $[a\star b]=[f(a)]\odot [b]$ applied several times. Indeed, notice that  
 $[{\tilde e}_{2}'(a,b)]=[a\star (a\star b)]=[f(a)]\odot [a\star b]=[f(a)]\odot ([f(a)]\odot [b])={\bar e}_{2}'([a], [b])$, and continuing in this way we get that 
\[{\bar e}_{i+1}'([f(a)],[b])=[{\tilde e}_{i+1}'(a,b)].\]
 
Observe that the result  does not depend on the particular function $\star $ which was chosen ($\star $ can be chosen in several different ways as it depends on function $\rho ^{-1}$). So if we choose a different function $\rho ^{-1}$ and construct operation $\star $ using this $\rho ^{-1}$, the element $[{\tilde e}_{i}'(a,b)]$
 is the same.
 
Observe also that \[{\bar e}_{i}'([f(a)],[b]+[b'])={\bar e}_{i}'([f(a)],[b+b'])=[{\tilde e}_{i}'(a,b+b')]=[{\tilde e}_{i}'(a,b)]+[{\tilde e}_{i}'(a,b')]=\]
\[={\bar e}_{i}'([f(a)],[b])+{\bar e}_{i}'([f(a)],[b']).\]
\end{remark}
\begin{definition} \label{ba2}  Let assumptions be as in Lemma \ref{super7}. Let notation be as in Definition \ref{bar}. 
We  define for $b\in pA$, \[{\bar e}_{0}([a], [b])={\bar e}_{\frac {p-1}2}([a], [\rho ^{-1}(b)]),\]
 
\end{definition}
  Observe that  \[{\bar e}_{0}([f(a)],[b])=[{\tilde e}_{0}(a,b)]\] for each $b\in pA$.  It follows because by the first part of the proof  we have  \[{\bar e}_{i}'([f(a)],[c])=[{\tilde e}_{i}'(a,c)]\] for all $c\in A$, and we can take $c=\rho ^{-1}(b)$ and the result does not depend of the choice of function $\rho ^{-1}$ (since ${\tilde e}_{i}'$ does not depend on function $\rho ^{-1}$).

\begin{lemma}\label{p}  Let assumptions be as in Lemma \ref{super7}. Suppose that pseudobrace $A$ satisfies property $2$. 
 Let notation be as in  Definition \ref{bar},  then 
 \[{\bar e}_{\frac {p-3}2}'([a],[A])\subseteq p[A]\] and \[{\bar e}_{\frac {p-3}2}'([a],[ann(p^{2k-1})])=[0].\]
\end{lemma}
\begin{proof}  Let $a,b, c\in A$. By Corollary \ref{tilde} we have 
 \[{\tilde e}_{\frac {p-3}2}'(c,b)\subseteq pA.\] 
 Let $g(a)$ be as in Lemma \ref{g(a)}. We can take \[c=g(a)\] to obtain:
\[{\tilde e}_{\frac {p-3}2}'(g(a),b)\subseteq pA.\]
 By Remark \ref{2} we get 
\[{\bar e}_{i}'([f(c)],[b])=[{\tilde e}_{i}'(c,b)].\]
 Therefore, for $i\geq {\frac {p-3}2}$ we have 
\[{\bar e}_{i}'([a],[b])=[{\tilde e}_{i}'(c,b)]\subseteq [pA].\]
 This concludes the proof.
  Similarly, we have ${\tilde e}_{\frac {p-3}2}'(c,ann(p^{2k}))\subseteq ann(p^{2k})$ and this implies ${\bar e}_{\frac {p-3}2}'([a],[ann(p)])=[{\tilde e}_{\frac {p-3}2}'(c, ann(p^{2k-1}))]=[0].$
\end{proof}
 Reasoning similarly as in Lemma \ref{p} we obtain the following corollary:
\begin{corollary}  Let assumptions be as in Lemma \ref{super7}. Let notation be as in  Definition \ref{bar}, 
then the definition of ${\bar e}_{0}([a], [b])$ does not depend on the particular function $\rho ^{-1}$ becuse ${\bar e}_{0}([a],[b])=[{\tilde e}_{0}(g(a),b)]$ and ${\tilde e}_{0}$ does not depend on $\rho ^{-1}$. Moreover, ${\bar e}_{0}([a],[b]+[c])={\bar e}_{0}([a],[b])+{\bar e}_{0}([a],[c])$. 
\end{corollary}

{\em Remark.} Let $b\in pA$, $d, d'\in A$,  and let $[d]$, $[d']$ be such that $p[d]=[b]$  then $d-\rho^{-1}(a)\in ann(p^{2k-1}$.   Let assumptions be as in Lemma \ref{super7}. Let notation be as in Definition \ref{bar}. 
 By Lemma \ref{p} we have \[{\bar e}_{0}([a], [b])={\bar e}_{\frac {p-1}2}([a], [d]).\]

  Therefore,  \[{\bar e}_{0}([a], [b])={\bar e}_{\frac {p-1}2}([a], \rho '^{-1}([b])),\]
   where $\rho '^{-1}:p[A]\rightarrow [A]$ is an arbitrary pullback function on $[A]$ (defined as in Section \ref{pullback} but for brace $[A]$ instead of $A$). 
\begin{theorem}\label{clever} 
  Let assumptions be as in Lemma \ref{p}.  Denote $[a]=[a]_{ann (p^{2k})}$. Let $f(a)$ be defined as in Definition \ref{diam}. 
 Define sets ${\bar E}_{i}^{t}([a], [b])$ for $a,b\in A$ analogously like sets $E_{i}^{t}(a,b)$ but using elements ${\bar {e}}_{i}([a], [b])$ instead of  elements $e_{i}(a,b)$ to define them.  
 Then  
\[[a*b]\in [f(a)]\odot [b]+{\bar E}_{1}^{1}([f(a)],[b])+{\bar E}_{2}^{0}([f(a)],[b])\]
and 
 \[{\bar E}_{i}^{t}([f(a)],[b])=[{\tilde E}_{i}^{t}(a,b)].\]
Moreover, this presentation depends only  on the additive group  of brace $(A, + )$ and does not depend on elements $a,b$.
\end{theorem}
\begin{proof} By Definitions \ref{bar} and \ref{ba2} we have \[{\bar e}_{i}([f(a)],[b])=[{\tilde e}_{i}(a,b)].\]
 By Definition \ref{ba2} \[{\bar e}_{0}([f(a)],[b])=[{\tilde e}_{0}(a,b)].\]
 It follows that  \[{\bar E}_{i}^{t}([f(a)],[b])=[{\tilde E}_{i}^{t}(a,b)].\]
 The result now follows from Lemma \ref{super7}.
\end{proof}

\begin{lemma}\label{7*}  Let assumptions be as in Lemma \ref{clever}.
Then 
\[[a]\in [f(a)]+{\bar E}_{1}^{0}([f(a)], [f(a)]).\]
  As usual, by $[a]$ we mean $[a]_{ann(p^{2k})}$. 
 \end{lemma}
\begin{proof} By Lemma \ref{6*} we have 
\[a\in f(a)+{\tilde E}_{1}^{0}(a, f(a)).\]

Theorem \ref{clever}  applied for $b=f(a)$ gives the required conclusion. 
\end{proof}
\begin{theorem}\label{mily}  Let assumptions be as in Lemma \ref{clever}.
  Let $[a]=[a]_{ann (p^{2k})}$. 
 Then,  
\[[a*b]\in [a]\odot [b]+w([a], [b])\]
 where $w([a],[b])$ is an element obtained by applying
 operations $+$, $\odot $, ${\bar e}_{0}$
 to some copies of elements $[a], [b]$ and the order of applying these operation depends only on the cardinality of $A$. Recall also that ${\bar e}_{0}$ 
 does not depend of the choice of function $\rho ^{-1}$ used to construct it. Moreover ${\bar e}_{0}([x],[y])$ is applied for $x\in A$,  $y\in pA$.  
\end{theorem}
\begin{proof} Let $g(a)$ be as in Lemma \ref{g(a)}. We know that $[f(g(a))]=[a]$.
 By Lemma \ref{7*} we applied to $g(a)$ instead of $a$ we get 
 \[[g(a)]\in [a]+{\bar E}_{1}^{0}([a], [a]).\]
Note that it does not depend of the choice of operation $\rho ^{-1}$.
 We will now repeat analogous steps. Each of these steps does not depend on the choice of operation $\rho ^{-1}$, and at each of step $\rho ^{-1}$ is well defined as it acts on element from $pA$ (as by the definition of sets ${\tilde E}_{i}^{t}(a, b)$ the operation $\rho^{-1}$ is always well defined as it always acts on elements from $pA$ during the construction of these sets).
  
 Applying the same argument to $g(a) $ instead of $a$ and simplifying using the above expression we obtain:
 \[[g(g(a))]\in [a]+w_{1}([a])\]
 where $w_{1}([a])$ is an element obtained by applying operations $+, \odot $, ${\bar e}_{0}$  to several copies of element $[a]$ several times.  
The sequence of steps taken to obtain $w_{1}([a])$ depends only  on the additive group  of brace $(A, + )$ and does not depend on elements $a,b$ and $\rho $ and  operation $\odot $. 
 Indeed, note that \[[g(g(a))]\in [g(a)]+{\bar E}_{1}^{0}([g(a)], [g(a)])\subseteq [a]+{\bar E}_{1}^{0}([a], [a])+ 
{\bar E}_{1}^{0}([a]+{\bar E}_{1}^{0}([a], [a]), [a]+{\bar E}_{1}^{0}([a], [a])).\]

Recall that $A$ has cardinality $p^{n}$. 
Continuing on in this way we obtain
\[[g^{p^{n}!-1}(a)]\in [a]+w_{p^{n!}-2}([a])\]
  where $w_{n!-2}([a])$ is an element obtained by applying operations $+, \odot , {\bar e}_{0}$ to several copies of element $[a]$.
 By Lemma \ref{g(a)}
$[g^{p^{n}!-1}(a)]=[f[a])$ therefore 
\[[f(a)]\in [a]+w_{n!-2}([a])\]
 We can now apply this expression  $[a]+w_{p^{n}!-2}([a])$ instead of $f(a)$ in Theorem \ref{clever}, and this concludes the proof.
\end{proof}

\begin{theorem}\label{nareszcie1}
 Let notation be as in Lemma \ref{clever}. Denote $[a]=[a]_{ann (p^{2k})}$. 
 Let $a,b\in A$, then $[a]*[b]$ is obtained by applying
 operations $+$, $\odot $, $\rho ^{-1}$ to elements $[a]$ and $[b]$ and the order of applying these operation depends only on the additive group of $A$. The order of applying these operations only depends on the cardinality of $A$ and the result
 does not depend of the choice of function $\rho ^{-1}$ used to construct it (and $\rho ^{-1}$ is always applied to elements from $pA$). 
\end{theorem}
\begin{proof} Let $a,b\in A$.
 Observe that $[a]\odot  [b]=
[\wp^{-1}(\wp^{-1}((p^{k}a)*(p^{k}b)))].$ 
 It can be verified by multiplying both sides by $p^{2k}$ (since if $u,v\in A$ then  $[u]=[v]$ if and only if $u-v\in ann(p^{2k})$ which is equivalent to $p^{2k}u=p^{2k}v$). 
Notice that $[\wp^{-1}(\wp^{-1}((p^{k}a)*(p^{k}b)))]$  does not depend on the choice of the function $\wp ^{-1}$ since $[ann(p^{2k})]=[0]$. 
 Notice that $[a]\odot  [b]=
[(\rho^{-1})^{2k}((p^{k}a)*(p^{k}b)))].$ 
 The result now follows from Theorem \ref{mily}.
\end{proof}

{\em Remark.} For $\rho^{-1}[x]$ we  can take $[\rho ^{-1}(x)$] since the result does not depend on the choice of $\rho ^{-1}$ (as we can define another $\rho $ function by  $\rho _{2}([x])=[\rho ^{-1}(x)$]). Moreover,  at each step we can use a different function $\rho ^{-1}$ and the result would be the same.
  
$ $

{\bf Proof of Theorem \ref{71}}. It follows from Theorem \ref{nareszcie1}. 

$ $

 The following result was proved in \cite{paper1}. Notice that Property $1$ was called Properties $1'$ and $1''$.
$ $

\begin{theorem}\label{f(a)}(Theorem $8$, \cite{paper1}). Let $A$ be a brace. Assume that $ann(p^{i})$, $p^{i}A$ are ideals in $A$ for every $i$ and that $p^{i}A*ann(p^{j})\subseteq ann(p^{j-i})$ for all $j\geq i$.
 Suppose that $A$ satisfies property $1$. Let $k$ be a natural number.  Then the map $[a]_{ann (p^{2k})}\rightarrow [f(a)]_{ann (p^{2k})}$ is well defined and injective, where $f(a)$ is defined as in Lemma \ref{j}.  
\end{theorem}
 Notice that the proof of Theorem \ref{f(a)} in \cite{paper1} works for arbitrary $k$ (it is not necessary to assume that $p^{k(p-1)}A=0$).

\begin{proposition}\label{5}(Lemma $10$,  \cite{paper1}).  Let $A$ be a brace of cardinality $p^{n}$ for some prime number $p$ and some natural number $n$.
 Suppose that $A$ satisfies property $1$. Then  $ann(p^{i})$, $p^{i}A$ are ideals in $A$ for every $i$ and that $(p^{i}A)*ann(p^{j})\subseteq ann(p^{j-i})$ for all $j\geq i$.
\end{proposition}

$ $

{\bf Proof of Theorem \ref{main}.}  By assumptions $A$ satisfies property $1$.  By Theorem \ref{f(a)} and Lemma \ref{5}, $A$ satisfies property $2$ (since $A$ satisfies property $1$). Therefore, $A$ satisfies assumptions of Theorem \ref{nareszcie1}. 
 We obtain that for $a,b\in A$ we have   $[a*b]=q([a],[b])$ where $q([a],[b])$ is the formula from Theorem \ref{nareszcie1} which is obtained by applying operations $\odot $, $+$ $\rho ^{-1}$ on brace $A'$. 
  Recall that 
 $[a]\odot [b]=[\wp^{-1}((p^{k}a)*b)]$.
 Observe that $[a]\odot [b]=[\wp^{-1}(wp^{-1}((p^{k}a)*(p^{k}b))]$. So $[a]\odot [b]$ depends only on the operation $*$  in $p^{k}A$, and operation $+$ in $A$ and it is obtained by applying  a pullback function $\wp ^{-1}$ and 
 the result does not depend on the function $\wp^{-1}$. 
 The result now  follows from Theorem \ref{nareszcie1}.

$ $

{\bf Acknowledgments.} The author acknowledges support from the
EPSRC programme grant EP/R034826/1 and from the EPSRC research grant EP/V008129/1. 
 The author is very grateful to Bettina Eick and Efim Zelmanov for answering her questions about what is known about extensions of Lazard's correspondence in group theory.

\pagebreak

$ $

$ $

\begin{center}
\title{\huge Braces whose additive groups have small rank}
\end{center}

$ $

\begin{center}
\author{ Agata Smoktunowicz}
\end{center}

$ $

\date{ }
\maketitle
\begin{abstract}  
Let $A$ be a brace of cardinality $p^{n}$ for some prime number $p>3$, and let 
$k$ be such that $p^{k(p-1)}A=0$.  Denote $ann(p^{4})=\{a\in A:p^{4}a=0\}$.
 Suppose that the additive group of $A$ has rank at most ${\frac {p-1}4}$. 
 It is shown that brace $A/ann(p^{4k})$ is obtained by the Agrachev-Gamkerlidze  construction of the group of flows from a pre-Lie ring, where instead of division by $p$ a function $\rho ^{-1}:pA\rightarrow A$  is used.
 This formula only depends on the cardinality of $A$ and it does not depend on the choice of $\rho ^{-1}$.  
Similar results for some other types of braces are also presented.
\end{abstract}

\section{Introduction}
 In 2007, Wolfgang Rump introduced the notion of a brace in order to describe all non-degenerate involutive set-theoretic solutions of the Yang-Baxter equation. Braces are also related to many topics in algebra and mathematical physics;  
 for example, braces are exactly bijective 1-cocycles on abelian groups \cite{Rio}, so they have connections to homological algebra and group theory. 
They also have strong connections to Lie algebras, cohomology of groups and  cohomology of algebras. 
Quantum integrable systems related to braces were
introduced several years ago, and braces and skew-braces are also connected to reflection and Pentagon equations \cite{pent, doikou}.

 Braces whose additive groups have a small rank were investigated by David Bachiller in \cite{BA}, and he proved several interesting results about such braces, for example that additive and multiplicative orders of each element in such a brace agree. Recall also that an abelian group of cardinality $p^{n}$ for a prime number $p$ and a natural number $n$ has {\em rank} at most $c$  if it is a direct sum of at most  $c$ cyclic groups.

Because braces are a generalisation of Jacobson radical rings, they can be investigated by applying methods from noncommutative ring theory, and we will use these methods in this paper. 
  Methods from noncommutative ring theory were  used in several other papers on related  topics \cite{an, al, TB, cjo, Rio, CO, KS}. 
In particular, in \cite{Facchini}   Cerqua and  Facchini  applied noncommutative ring theory methods  to describe connections between braces and pre-Lie algebras.

 The main results of this paper are as follows.

\begin{theorem}\label{rank}
 Let $(A, \circ , +)$ be a brace of cardinality $p^{n}$ for a prime number $p$ and for a natural number $n$. Suppose that the additive group $(A, +)$ of brace $A$ has rank not exceeding ${\frac {p-1}4}$.
 Let $k$ be such that $p^{k(p-1)}A=0$.
 Then the factor brace  $A/ann(p^{4k})$ is obtained by the construction of modified group of flows 
 applied to a pre-Lie ring with the same additive group.
 \end{theorem} 

 The formula for a modified group of flows used  in Theorems \ref{rank},  \ref{uniform} and \ref{main1}  is introduced in Section \ref{groupflows} and  is the same as the formula for the group of flows in \cite{AG},  but instead of a division by $p$ a function $\rho ^{-1}$ is used resembling division by $p$. Moreover the result does not depend on the choice of function $\rho ^{-1}$.
 
 Let $(A,\circ , +)$  be a brace, then $a*b=a\circ b-a-b$.  Recall that $A^{1}=A$ and inductively $A^{i+1}=A*A^{i}$ is the left nilpotent series of $A$. 
 The next result of this paper is:
\begin{theorem}\label{uniform}
Let $(A, +, \circ )$ be a brace of cardinality $p^{n}$ for a prime number $p$ and for a natural number $n$, whose additive group $(A, +)$ is a direct sum of some number of the cyclic group of cardinality $p^{\alpha }$ for some $\alpha $. 
  Suppose that $A^{j}\in pA$ for some $j\leq  {\frac {p-1}4}$. Let $k$ be such that $p^{k(p-1)}A=0$.

 Then the factor  brace  $A/ann(p^{4k})$ is obtained by the construction of modified group of flows applied to a pre-Lie ring with the same additive group.
 \end{theorem} 
 It is known that if $A$ is a brace satisfying the assumptions of either Theorem \ref{rank} or of Theorem \ref{uniform} then $ann(p^{i})=\{a\in A:p^{i}a=0\}$ is an ideal in $A$ for every $i$ (see \cite{paper1}).

 More details about Theorems \ref{rank} and \ref{uniform} can be found in  Theorems \ref{rank1} and  \ref{uniform1}. We  also prove some more general results,  mainly Theorem \ref{main2}. We first introduce Properties $1'$ and $1''$. 

Let  $A$ be a brace of cardinality $p^{n}$ for a prime number $p>2$ and a natural number $n$. We say that $A$ satisfies property $1'$ if  $A^{\lfloor{\frac {p-1}4}\rfloor }\in pA$.
 We also consider braces satisfying the following Property $1''$: 
  Denote $ann(p^{i})=\{a\in A: p^{i}a=0\}$. We say that $A$ satisfies property $1''$ if   $A*(A*\cdots A*ann(p^{i}))\subseteq p\cdot ann(p^{i})$  
for $i=1,2, \ldots $, 
 where $A$ appears $\lfloor {\frac {p-1}4}\rfloor$ times in this expression.

\begin{theorem}\label{main2} Let $A$ be a brace of cardinality $p^{n}$ for some prime number $p$ and a natural number $n$. Suppose that $k$ is such that $p^{k(p-1)}A=0$.
Suppose that  $A$ satisfies properties $1'$ and $1''$.
 Then the brace $A/ann(p^{4k})$ is obtained from the left nilpotent  pre-Lie ring $(A/ann(p^{4k}), +, \bullet)$ obtained in Theorem \ref{main} by the construction of modified group of flows.
 \end{theorem}
 Observe that the assumptions of Theorem \ref{main} are stronger than 
 in the main result of \cite{paper1}, but we also obtain the stronger result than in \cite{paper1}, as we have the concrete formula for the passage from pre-Lie rings to braces, and this formula is 
 the same as the formula for the group of flows in \cite{AG}  but instead of a division by $p$ a function $\rho ^{-1}$ is used resembling division by $p$. Moreover, the result does not depend on the choice of function $\rho ^{-1}$. 
 See also \cite{M}, \cite{passage} for a description of this group in a notation which will be used in this paper.

 The above result shows connections between some types of braces  and pre-Lie rings outside of Lazard's correspondence. 

The idea of applying the pullback function $\rho ^{-1}$ was inspired by papers \cite{lazard, shalev} where another formula depending on $p$, namely $(a^{p}\circ b^{p})^{1/p}$, appeared in the context of group theory.
 Connections between finite braces and pre-Lie rings inside the context of Lazard's correspondence were investigated in \cite{rump,  BA, passage, Sheng, Senne, Doikou}. However, none of these papers considered the case of finite braces  outside of the context  of Lazard's correspondence. 
 In \cite{shalev}, Aner Shalev constructed Lie rings associated to uniform groups outside of the Lazard's correspondence. 

 We will now give an outline of the paper. Section $2$ contains basic background information on braces and pre-Lie rings. Section $3$ contains background notation. It contains the main definitions for the  notation used in the paper. Section $3$ has the following subsections:
(3.1) Pullback introduces function $\wp ^{-1}$, which is similar to division by $p^{k}$. It also introduces function $\rho ^{-1}$ which resembles division by $p$. (3.2) Pseudobraces. This section introduces pseudobraces, a generalisation of braces. When we apply a modified group of flows to a pre-Lie ring we obtain a pseudobrace \cite{paper1}.
Every brace is a pseudobrace. (3.3) Section (3.3) introduces operation $\odot $ for braces and pseudobraces. We have $[a]\odot [b]=[\wp^{-1}((pa)*b)]$.
(3.4), (3.5) and (3.6) In these sections we introduce properties $1$ and $2$ for braces, which are   used as assumptions throughout the paper. We also prove some results about these properties. We also review properties $1'$ and $1''$ from the introduction.
 Section (4) introduces a generalisation of the construction obf the group of flows from \cite{AG}. In section (5) it is shown that generalised groups of flows are pseudobraces.
 In Section $6$ we prove the main results of the paper.

\section{Background information } 
 We will use the same notation as in \cite{passage, asas, paper1}. We recall this notation below for the convenience of the reader:
\subsection{Braces}
A set $A$ with binary operations $+$ and $* $ is a {\em  left brace} if $(A, +)$ is an abelian group and the following  version of distributivity combined with associativity holds.
  \[(a+b+a*b)* c=a* c+b* c+a* (b* c), \space  a* (b+c)=a* b+a* c,\]
for all $a, b, c\in A$; moreover  $(A, \circ )$ is a group, where we define $a\circ b=a+b+a* b$.
In what follows we will use the definition in terms of the operation `$\circ $' presented in \cite{cjo} (see \cite{Rump}
for the original definition): a set $A$ with binary operations of addition $+$ and multiplication $\circ $ is a brace if $(A, +)$ is an abelian group, $(A, \circ )$ is a group and for every $a,b,c\in A$
\[a\circ (b+c)+a=a\circ b+a\circ c.\]
  All braces in this paper are left braces, and we will just call them braces.
 
Let $(A, +, \circ )$ be a brace.  Recall that $I\subseteq A$ is an ideal in $A$ if for
$i,j\in I$, $a\in A$ we have $i+j\in I, i-j\in I, i*a, a*i\in I$ where $a*b=a \circ b-a-b$.

Let $A$ be a brace and $a\in A$ then we denote $a^{\circ j}=a\circ a\cdots \circ a$
  where $a$ appears $j$ times in the product on the right hand side.  By $A^{\circ p^{i}}$ we will   denote the subgroup of the multiplicative group $(A, \circ)$ of brace $A$ generated by $a^{\circ p^{i}}$ for $a\in A$. 
  In general for a set $S\subseteq A$, by $S^{\circ {p^{i}}}$ we denote the subgroup of $(A, \circ )$ generated by elements $s^{\circ {p^{i}}}$ for $s\in S$. 

Let $A$ be a brace. Let $B, C\subseteq A$, then by $B*C$ we define the additive subgroup of $A$ generated by 
 elements $b*c$ where $b\in B$, $c\in C$.

 Recall that $A^{1}=A$ and inductively $A^{i+1}=A*A^{i}$ is the left nilpotent series of $A$. This series was introduced in \cite{Rump} by Wolfgang  Rump, and 
 he showed that if $A$ is a brace of cardinality $p^{n}$ for a prime number $p$ and a natural number $n$ then $A^{n+1}=0$. 

  Recall that if $I$ is an ideal in the brace $A$ then the factor brace $A/I$ is well defined.
The elements of the brace $A/I$ are cosets $[a]_{I}:= a + I=\{a+i:i\in I\}$ where $a \in A$, which we will simply denote by $[a]_{I}$, so $[a]_{I} =[b]_{I}$ if and only if $a-b\in I$.
 In all factor braces we will use operations $+, \circ $  from brace $A$ to denote operations in the factor brace $A/I$ (with a slight abuse of notation).

   Let $(A, +, \circ )$ be a brace. One of the mappings used in connection with braces are the  maps $\lambda _{a}:A\rightarrow A$ for $a\in A$. Recall that  for $a,b, c\in A$ 
   we have $\lambda _{a}(b) = a \circ  b - b$, $\lambda _{a\circ c}(b) = \lambda _{a}(\lambda _{c}(b))$.

\subsection{Pre-Lie rings}\label{pre-Lie}\label{preLi}
 Recall that a   {\em pre-Lie ring} $(A, +, \cdot)$ is  an abelian group $(A,+)$ with a binary operation $(x, y) \rightarrow  x\cdot y$ such that 
\[(x\cdot y)\cdot z -x\cdot (y\cdot z) = (y\cdot x)\cdot z - y\cdot (x\cdot z)\]
and $(x+y)\cdot z=x\cdot z+y\cdot z, x\cdot (y+z)=x\cdot y+x\cdot z,$
 for every $x,y,z\in A$. A pre-Lie ring $A$  is {\em  nilpotent} or {\em strongly nilpotent}   if for some $n\in \mathbb N$ all products of $n$ elements in $A$ are zero. We say that $A$ is {\em left nilpotent} if for some $n$, we have $a_{1}\cdot (a_{2}\cdot( a_{3}\cdot (\cdots  a_{n})\cdots ))=0$ for all  $a_{1}, a_{2}, \ldots , a_{n}\in A$.

We will now use the construction of pre-Lie ring associated to brace $A$ from \cite{paper1}.
 We will use $\bullet $ from \cite{paper1} multiplied by $-(1+p+\cdots+ p^{n})$, as it also gives a pre-Lie multiplication on $A/ann(p^{2k})$. Notice that $p^{n}A=0$ by the Lagrange theorem applied to the additive group $(A, +)$.

 The following Theorem immediately follows from Theorem $16$, \cite{paper1}.

\begin{theorem} \label{preLie}  
Let $A$ be a brace of cardinality $p^{n}$ for a prime number $p$ and a natural number $n$. Suppose moreover that $ann(p^{i})$, $p^{i}A$ are ideals in $A$ for every $i$ and that $(p^{i}A)*ann(p^{j})\subseteq ann(p^{j-i})$ for all $j\geq i$.  
We define pre-Lie ring associated to $A$ similarly as in Theorem $16$,  \cite{paper1}.
 We define:

 \[[a]\bullet [b]=-(1+p+\ldots +p^{n})\sum_{i=0}^{p-2}\xi ^{p-1-i}((\xi ^{i}p^{k}a)* b),\]
where $\xi =\gamma  ^{p^{n-1}}$ where $\gamma $ is a primitive root modulo $p^{n}$.
Then  $(A/ann(p^{2k}),+, \bullet )$ is a left nilpotent pre-Lie ring. \end{theorem}
\begin{proof} The pre-Lie ring property  immediately follows from  Theorem $16$ \cite{paper1}.
 The fact that the obtained pre-Lie ring is left nilpotent immediately follows from Proposition $18$, \cite{paper1}. 

\end{proof}

\section{Background notation}

\subsection{Pullback}\label{pullback}
In this section we  will use the same definition of a pullback as in \cite{paper3}. We recall the definition for the conveniece of the reader:  
Let $A$ be a brace of cardinality $p^{n}$ for some prime number $p$ and some natural number $n$. 
Fix a natural number $k$.  
For $a\in p^{k}A$ let $\wp^{-1}(a)$ denote an element $x\in A$ such that $p^{k}x=a$. Such an element may not be uniquely determined in $A$, but we can fix for every $a\in p^{k}A$ such an element $\wp^{-1}(a)\in A$.
Notice that $p^{k}(\wp^{-1}(a))=p^{k}x=a$.

 For $A$ as above, $ann(p^{i})=\{a\in A:p^{i}a=0\}$. 
 We recall a result from \cite{paper1}.
\begin{lemma}\label{33} Let $A$ be a brace and let $I$ be an ideal in $A$. Let $A/I$ be defined as above. Assume that $ann(p^{k})\subseteq I$. 
Let $\wp^{-1} : p^{k}A \rightarrow A$ be defined as above. Then, for $a, b\in p^{k}A$ we have
$[\wp^{-1}(a)]_{I} + [\wp^{-1}(b)]_{I} = [\wp^{-1}(a + b)]_{I}.$ 
This implies that for any integer $m$ we have
 $[m \wp^{-1}(a)]_{I} = [\wp^{-1}(ma)]_{I}.$
\end{lemma}

 As per \cite{paper3}, in this paper we will assume that 
\[\wp^{-1}(a)=(\rho ^{-1} )^{k}(a)=\rho^{-1}( \cdots (\rho^{-1}(a)\cdots )),\] where $\rho ^{-1}:pA\rightarrow A$ is a given function $\rho ^{-1}:pA\rightarrow A$ such that $p\cdot \rho^{-1}(x)=x$, for $x\in pA$.
 Observe that
\[\rho ^{-1}(a+b)-\rho^{-1}(a)-\rho ^{-1}(b)\in ann(p),\]
for all $a, b\in A$.
 Lemma \ref{33} holds for function $\rho ^{-1}$ when we assume that  $k=1$.

\subsection{Pseudobraces} 
In this section we  introduce pseudobraces. All braces are pseudobraces.

\begin{definition}\label{pseudobrace}
Let $A$ be a set, and let $+$ and $\circ $ be binary operations of $A$. Let $a,b\in A$. We denote $a\circ b=a*b-a-b$.  
We denote  $a^{1}=a$ and inductively $a^{\circ  i+1}=a\circ a^{\circ  i}$.
We say that $A$ is a pseudobrace if the following properties hold:
\begin{enumerate} 
\item $(A,+)$ is an abelian group of cardinality $p^{n}$ for some prime number $p$ and some natural number $n$. 
\item For every $a\in A$ there is an element in $A$, denoted as $a^{\circ  -1}$, such that $a\circ a^{\circ -1}=0$, where $0$ is the identity in the group $(A,+)$.
 We denote $a^{\circ -2}=a^{\circ -1}\circ a^{\circ -1}$, $a^{\circ -3}=a^{\circ -1}\circ a^{\circ -2}$, etc..
\item For all integers $i,j$ and every $a,b\in A$, $(a^{\circ  i}\circ a^{\circ j})\circ b=a^{ \circ i+j}\circ b$ 
 and $a^{\circ i}\circ a^{\circ j}=a^{\circ i+j}$, $a^{circ i}\circ b=a\circ (a\circ \cdots a\circ b))$ where $a$ appears $i$ times on the right-hand side of this equation. 
\item $A^{n+1}=0$ where $A^{1}=A$ and inductively $A^{i+1}=A*A^{i}$.  
\item $a*(b+c)=a*b+a*c$ for all $a,b,c\in A$.
\end{enumerate}
\end{definition}

$ $

 The group $(A, +)$ is called the additive group of pseudobrace $A$.

$ $

For a pseudobrace $A$ we define $p^{i}A=\{p^{i}a:a\in A \}$, $ann(p^{i})=\{a\in A:p^{i}A=0\}$.  

$ $

We say that a subset $I\subseteq A$ is an {\em ideal} in a pseudobrace $A$ if 
for $i,j\in I$, $a,b\in A$ we have $i+j\in I, i-j\in I$, $a*i, i*a\in I$ and $(a+i)*(b+j)-a*b\in I$. Observe that this implies that the factor pseudorace $A/I$ is well defined and elements of 
$A/I$ are sets $[a]_{I}=\{a+i:i\in I\}$.  We have $[a]_{I}+[b]_{I}=[a+b]_{I}$, $[a]_{I}*[b]_{I}=[a*b]_{I}$. 

Note that if $A$ is a brace this definition coincides with the normal definition of an ideal in a brace, because for an ideal in a brace the factor brace is well defined \cite{BaPhd, Rump}.

$ $

Let $b\in A$. We say that a subset $L(b)\subseteq A$ is an {\em left ideal} in a pseudobrace $A$ generated by $b$ if 
 $b\in L(b)$ and for for $i,j\in L(b)$ and every  $a\in A$ we have $i+j\in L(b), i-j\in L(b)$, $a*i\in L(b)$.

 We do not know the answer to the following question: 
\begin{question} 
Is every pseudobrace a brace?
\end{question}

\subsection{ Operation $\odot $ for braces and pseudobraces}

In this section we define operation $\odot $ for braces and pseudobraces.

Let $A$ be a pseudobrace of cardinality $p^{n}$ and let $k$ be a natural number. Note that if $ann(p^{2k})=\{a\in A: p^{2k}a=0\}$  is an ideal in the pseudobrace $A$ then the factor brace $A/ ann(p^{2k})$ is well defined.
The elements of the pseudobrace $A/I$ are cosets $[a]_{I}:= a + I=\{a+i:i\in I\}$ where $a \in A$. In this section we denote for $a\in a$,  $[a]_{ann(p^{2})}=[a]$.
  
The following result is a generalisation of Theorem $8.1$, \cite{asas}.
\begin{theorem}\label{odot}
Let $(A, +, \circ )$ be a pseudobrace of cardinality $p^{n}$ for some prime number $p$ and some natural number $n$. Let $k$ be a natural number such that $p^{k(p-1)}A=0$. Assume that $p^{k}A$, $ann(p^{k})$, $ ann(p^{2k})$ are ideals in $A$ and $(p^{k}A)*ann(p^{2k})\subseteq ann(p^{k})$. Let $\wp^{-1}:p^{k}A\rightarrow A$ be defined as in the section \ref{pullback}. Let $\odot $ be defined as 
\[[x]\odot [y]=[\wp^{-1}((p^{k}x)*y)].\] Then $\odot $ is a well defined operation on $A$.
\end{theorem}
\begin{proof} The same proof as in \cite{paper1} works.
\end{proof}

\subsection{Property $1$}\label{property1}
  We start with recaling the definition of Property $1$ from \cite{paper3}. 
\begin{definition}
 Let $A$ be a pseudobrace of cardinality $p^{n}$ for some prime number $p>3$ and some natural number $n$. We will say that $A$ satisfies {\em property $1$} if 
\begin{enumerate}  
 \item for every $a, b\in A$, \[a*(a*( \cdots a*b))\in pA,\]
 \item and for every $a\in A$ and every $i>0$ 
\[a*(a*( \cdots a*ann(p^{i})))\in ann(p^{i-1})\]
\end{enumerate}
 where $a$ appears $\lfloor {\frac {p-1}4}\rfloor$ times in each of the  above expressions. 
\end{definition}

$ $ 
 
We now give examples of braces satisfying Property $1$.

We say that the additive group of a pseudobrace $A$ is {\em uniform} if it is a direct sum of some number of cyclic groups of cardinality $i $, for some $i $.

\begin{lemma}\label{s} Let $A$ be a pseudobrace of  cardinality $p^{n}$ for some prime number $p$ and some natural number $n$. Suppose  that the additive group of $A$ is uniform. Suppose also  that  for each $a,b\in A$  
  \[a*(a* \cdots a*b))\in pA\] where $a$ appears $\lfloor {\frac {p-1}4}\rfloor$ times in this expression. Then $A$ satisfies 
property $1$.
\end{lemma}
\begin{proof} Suppose that  the additive group of $A$ be a direct sum of cyclic 
 groups of cardinality $p^{\alpha }$ for some $\alpha $. Then $ann(p^{i})=p^{\alpha -i}A$, for each $i$. Therefore, 
   $ a*(a* \cdots a*ann(p^{i}))=a*(a* \cdots a*p^{\alpha -i}A)\in p\cdot p^{\alpha -i}A=p^{\alpha +1-i}A=ann(p^{i-1})$.
\end{proof}

\begin{lemma}\label{q} Let $A$ be a pseudobrace of cardinality $p^{n}$ for some prime number $p>3$ and some natural number $n$. Suppose that 
  for every $a, b\in A$, 
\[a*(a* \cdots a*b))\in pL(b)\] where $L(b)$ is the left ideal generated by $b$ in the pseudobrace $A$, and $a$ appears $\lfloor {\frac {p-1}4}\rfloor$ times in this expression. Then $A$ satisfies property $1$.  
\end{lemma}
\begin{proof} It follows because for $b\in ann(p^{i})$, we have $pL(b)\subseteq p\cdot ann(p^{i})\subseteq ann(p^{i-1})$.
\end{proof}
We say that the additive group of a pseudobrace $A$ has {\em rank} at most $c$  if it is a direct sum of at most  $c$ cyclic groups.

 In \cite{BA} Bachiller investigated braces whose additive groups have small rank. The following result in a generalisation of Lemma $2.7$ from \cite{BA}.
\begin{lemma}\label{r} Let $A$ be a pseudobrace of a cardinality $p^{n}$ for some prime number $p$ and a natural number $n$. Suppose  that the additive group of $A$ is has rank at most $\lfloor {\frac {p-1}4}\rfloor$.   
 Then $A$ satisfies 
property $1$. Moreover, $a_{1}*(a_{2}*( \ldots a_{c}*b))\in pL(b)$ for every $a_{1}, a_{2}, \ldots , a_{c}, b\in A$. Hence properties $1'$ and $1''$ hold.
\end{lemma}
\begin{proof} We will use a similar proof to the proof  of Lemma $2.7$ \cite{BA}. It suffices to show that $A$ satisfies assumptions of Lemma \ref{q}. 
Let $L_{a}: L(b)\rightarrow L(b)$ be given by $L_{a}(b)=a*b$.  
 Observe that $L(b)/pL(b)$ is a linear space over the field ${\mathbb Z}_{p}$
 of integers modulo $p$, the dimension of this linear space is at most $c$. Let $a, c\in A$, then  $L _{a}'(c+pL(b))=a*c+pL(b)$ is a linear map over the field ${\mathbb Z}_{p}$ (where $c\in L(b), c+pL(b)\in L(b)/pL(b)$).
  Observe that this is well defined, since in every pseudobrace $a*(pc)=a*(c+\ldots +c)=a*c+\ldots + a*c=p(a*c)$. 
 Define sets $A_{1}=L(b)$, and inductively $A_{i+1}=A*L_{i}$. Note that $A_{n+1}=0$ since $A^{n+1}=0$ (by the definition of pseudobrace).
 We will now choose a base of the linear space $L(b)/pL(b)$. Let first elements of this base span the set $A_{i}$ where $i$ is maximal possible. 
 We then add the next elements, such that the set of all so far choosen base elements spans  
 the set $A_{i-1}+A_{i}$. Continuing in this way we obtain a base of $L(b)/pL(b)$. 
 
 Observe that in this base $L_{a}'$  corresponds to a lower triangular  matrix. So $L_{a_{1}}'\ldots L_{a_{c}}'(b)=0$ for every $a_{1}, \ldots , a_{c}, b\in A$.  This follows because $L_{a}(A^{j})\subseteq A^{j+1}$ for every $j$.
 Consequently, $L_{a_{1}}'\ldots L_{a_{c}}'(b)\in pL(b).$
\end{proof}

\subsection{Properties $1'$ and $1''$. }
Let  $A$ be a pseudobrace of cardinality $p^{n}$ for a prime number $p>2$ and a natural number $n$.
 Then we define the following properties on the pseudobrace $A$:
\begin{itemize}
\item  We say that $A$ satisfies property $1'$ if  $A^{\lfloor{\frac {p-1}4}\rfloor }\in pA$. 
  Denote $ann(p^{i})=\{a\in A: p^{i}a=0\}$. We say that $A$ satisfies property $1''$ if    
$A*(A\cdots *(A*(ann(p^{i}))))\in p\cdot  ann(p^{i}),$
for $i=1,2, \ldots $  where $A$ appears $\lfloor {\frac {p-1}4}\rfloor$ times in this expression.
\end{itemize}
Note  that if a pseudobrace $A$ satisfies properties $1'$ and $1''$ then it satisfies property $1$.

\subsection{Property $2$}\label{property2}

Let $A$ be a pseudobrace of cardinality $p^{n}$ for some prime number $p>3$ and some natural number $n$. We fix natural number $k$ such that $p^{k(p-1)}A=0$. We recall the definition of Property $2$ from \cite{paper3}. 

\begin{definition} We say that  a pseudobrace $A$  satisfies Property $2$ if  for each $a,b\in A$, 
 we have that $p^{2k}a=p^{2k}b$ if and only if $p^{k}a^{\circ p^{k}}=p^{k}b^{\circ p^{k}}$. We also assume that  $p^{i}A$, $ann(p^{i})$ are ideals in $A$ for each $i$ and $(p^{i}A)*ann(p^{j})\subseteq ann(p^{j-i})$ for all $j\geq i$.
\end{definition}
The following lemma also appears in \cite{paper3} but we recall its short proof for a convenience of the reader.

\begin{lemma}\label{j} Let $A$ be a pseudobrace of cardinality $p^{n}$ for some prime number $p$ and some $n$.  Let $f:A\rightarrow A$ be a function  such that $p^{k}f(a)=a^{\circ p^{k}}$ for each $a\in A$.
  Then property $2$ implies that  the map $[a]_{ann (p^{2k})}\rightarrow [f(a)]_{ann (p^{2k})}$  is well defined and one-to-one.
\end{lemma}
\begin{proof}
 This follows from the fact that for $x,y\in A$, we have $[x]_{ann(p^{2k})}= [y]_{ann (p^{2k})} $ if and only if  
$x-y\in ann(p^{2k})$. This in turn is equivalent to $p^{2k}(x-y)=0$. So the map  $[a]_{ann (p^{2k})}\rightarrow [f(a)]_{ann (p^{2k})} $  is well defined and injective. Since  set $A/ann(p^{2k})$ is finite it follows that ths  is a bijection on $A/ann(p^{2k})$.
\end{proof}
  We recall a  Lemma from \cite{paper3}: 
 \begin{lemma}\label{g(a)}
 Let $A$ be a pseudobrace satisfying property $2$. Let notation be as in Lemma \ref{j}. For $a\in A$
 define \[g(a)=f^{(p^{n}!-1)}(a)\] where $f^{(1)}(a)=f(a)$ and for every $i$ we denote $f^{(i+1)}(a)=f(f^{(i)}(a))$. 
 Then \[[f(g(a))]=[g(f(a))]=[a].\]
  Moreover
  \[[f(a)]=[g^{(p^{n}!-1)}(a)]\]  where $g^{(1)}(a)=a$ and for every $i$ we denote $g^{(i+1)}(a)=g(g^{(i)}(a))$.
 \end{lemma}
 The following result was proved in \cite{paper1}.
\begin{proposition} \cite{paper1} Let $A$ be a brace of cardinality $p^{n}$ for some prime number $p$ and some natural number $n$.
 Suppose that $A$ satisfies property $1$. Then  $ann(p^{i})$, $p^{i}A$ are ideals in $A$ for every $i$ and that $(p^{i}A)*ann(p^{j})\subseteq ann(p^{j-i})$ for all $j\geq i$.
 Moreover, $A$ satisfies property $2$.
\end{proposition}

 We also  recall  the following result from \cite{paper3}: 
\begin{theorem}\label{71}
 Let $A$ be a pseudobrace which satisfies Property $1$ and Property $2$. Let $k$ be such that $p^{k(p-1)}A=0$.
 Let $\rho ^{-1}:pA\rightarrow A$ be any function such that $p\cdot \rho ^{-1}(x)=x$ for each $x\in A$.
 Let $a,b\in A$, then $[a]*[b]$ is obtained by applying
 operations $+$, $\odot $, 
 to elements $[a]$ and $[b]$ and $\rho ^{-1}$ to elements from $pA$, and the order of applying these operation depends only on the additive group of $A$ (it does not depend on $a$ and $b$ and it does not depend on any properties of $*$).  Moreover, the result
 does not depend of the choice of function $\rho ^{-1}$. 
\end{theorem}

\section{A modified construction of the group of flows}\label{groupflows}
 
In this section we discuss why  formulas obtained in the previous sections are related to the construction of the  group of flows from \cite{AG}.
 For a real number $q$, by $\lfloor q\rfloor $ we denote the largest integer not exceeding $q$.

\begin{definition}
 Let $A$ with operations $+, \cdot $ be a pre-Lie ring of cardinality $p^{n}$ for some prime number $p>3$ and some natural number $n$. For each $i$, denote $ann(p^{i})=\{a\in A: p^{i}a=0\}.$
 
We say that $A$ satisfies {\em Property $3$} if 
   for each $a_{1}, \ldots , a_{\lfloor{\frac {p-1}4}\rfloor }\in A$,
 \[{a_{1}}\cdot ({a_{2}}\cdot (\cdots {a_{\lfloor{\frac {p-1}4}\rfloor }}))\in pA,\]
 and 
 \[{a_{1}}\cdot ({a_{2}}\cdot (\cdots {a_{\lfloor {\frac {p-1}4\rfloor } }}\cdot (ann(p^{i}))))\in p\cdot ann(p^{i})\] for every $i\geq 1$.
\end{definition}

Denote \[L_{a}(b)=a\cdot b.\]
 Let $\rho ^{-1}:pA\rightarrow A$ be defined as any function such that $p\cdot \rho ^{-1}(x)=x$ for each $x\in pA$.  Note that this is the same as the definition of the  section \ref{pullback}. Note that the definition of $\rho ^{-1}$ depends only on the additive group $(A, +)$, so can be applied to a pre-Lie algebra instead of brace.

\begin{definition}\label{1} Let $p$ be a prime number and let $A$ be a pre-Lie ring satisfying Property $3$. 
 Let $j$ be a natural number, and let $s_{j}$ be the largest natural number such that $p^{s_{j}}$ divides $j!$. Observe that \[s_{j}\leq {\frac j{p}}+{\frac j{p^{2}}}+\cdots ={\frac j{p-1}}.\]
(as at  first we can count positive integers not exceeding $j$ which are divisible by $p$, then we count numbers divisible by $p^{2}$ and not exceeding $j$, etc.).
Denote \[\eta_{j}=\lfloor{\frac {(p-1)}4}\rfloor \cdot s_{j} .\]
 Note that \[L_{a}^{\eta_{j}}(b)\in p^{s_{j}}A\] and  $2{\eta }_{j}<j$.

For $s_{j}$ divisible by $p$ we
 denote \[E_{j,a}=L_{a}^{ j-\eta_{s_j}}(\rho ^{-1})^{s_{j}}L_{a}^{ \eta_{s_j}}.\] 
 If $j$ is not divisible by $p$ then $s_{j}=0$ and we denote  \[E_{j,a}=L_{a}^{ j}.\]

\end{definition}
\begin{remark}\label{2} Notice that  $E_{j, a}$ is well definied, and it does not depend on a choice of  function $\rho ^{-1}$. 
Indeed if we have two different functions $\rho _{1}^{-1}$ and $\rho _{2}^{-1}$ then 
 \[(\rho _{1}^{-1})^{s_{j}}L_{a}^{ e_{j}}-(\rho _{2} ^{-1})^{s_{j}}L_{a}^{ e_{j}}\in  ann (p^{s_{j}})\]
 for each $a,b\in A$.
By Property $3$, 
 \[L^{{s_{j}}\cdot {\lfloor {\frac {(p-1)}4}\rfloor}}(ann (p^{s_{j}}))=0,\]  and hence 
\[L^{j-\eta _{s_{j}}}(ann (p^{s_{j}}))=0,\]
 so $E_{j,a}$ is well defined and does not depend on the choice of the function $\rho ^{-1}$.

\end{remark}
\begin{remark}\label{czesto}
Notice also that  for every $a,b\in A$, $E_{j,a}(b)$ belongs to the ideal generated by $a$ in the pre-Lie ring $A$, since \[E_{j,a}(b)=a\cdot (L_{a}^{ j-\eta_{j}-1}((\rho ^{-1})^{s_{j}}(L_{a}^{ \eta_{j}}(b)))).\] 

\end{remark}
\begin{remark}\label{moving} Let notation be as in Remark \ref{2}.
Then  for all $i,j>0$ and all $a\in A, b\in p^{j}A$ we have
\[ L_{\eta _{j}+i,a}((\rho ^{-1})^{j}(b))=L_{\eta _{j}, a}((\rho ^{-1})^{j}L_{i, a}(b)).\]
 This follows from Property $3$ and since $L_{i,a}((\rho ^{-1})^{j}(b))-(\rho ^{-1})^{j}(L_{i,a}(b))\in ann(p^{j})$.

$ $

 To change the location of the function $\rho ^{-1}$  in the definition of $E_{j,a}$, a good approach is  to first move $\rho ^{-1}$ to the left, and leaving 
 $L_{{\lfloor {\frac {p-1}4}  \rfloor}, a}$
 between two occurences of $\rho ^{-1}$, and then each $\rho ^{-1}$ can be moved to the right.
\end{remark}

$ $

\subsection{Defining generalised group of flows }
 We will now generalise the construction of the group of flows from \cite{AG}. 
 The only difference to the original construction of group of flows from \cite{AG} is that the map $\rho ^{-1}:pA\rightarrow A$ will be used  instead of division by $p$. Notice, by Remark \ref{moving}
  we can move the function $\rho ^{-1}$ from one place to another in a given monomial, without changing the result, as long as there is  $L_{i, a}$  for $i\geq {\frac {p-1}4}$ before $\rho ^{-1}$ in this monomial. 
 
 In what follows we will modify the construction of the group of flows from \cite{AG}. This construction was later  mentioned in \cite{passage}, and we will use notation from \cite{passage}.  We start with a generalisation of function $W(a)$. 
Let $A$ be a pre-Lie ring satisfying property $3$.  We can add an identity $\upharpoonleft $ to this pre-Lie algebra, as in \cite{passage}. 
 
We denote $E_{1,a}(\upharpoonleft  )=a$ and $L_{1,a}(\upharpoonleft )=a$ and $L_{i+1,a}(\upharpoonleft )=a\cdot L_{i+1,a}(\upharpoonleft )$.
   We then denote, similarly as in Definition \ref{1},
  \[E_{j, a}(\upharpoonleft )=L_{a}^{ j-\eta_{s_j}}(\rho ^{-1})^{s_{j}}L_{a}^{ \eta_{s_j}}(\upharpoonleft ),\]  where $s_{j}$ and $\eta _{s_{j}}$ are as in Definition \ref{1}.
 Observe that by property $3$, $E_{j, a}(\upharpoonleft )\in A$, $E_{j, a}(b)\in A$, $E_{j,a}(ann(p^{i}))\subseteq ann(p^{i})$.  
 We are now ready to introduce function $W$. 

\begin{definition}\label{e} Let $p>3$ be a prime number, $n$ be a natural number and let $A$ be a pre-Lie ring of cardinality $p^{n}$.
 Let $A$ be a pre-Lie ring satisfying property $3$. Let $a,b\in A$.  
 Let $\gamma _{j}$ be integers not divisible by $p$ and such that 
 \[j!=\gamma _{j}\cdot p^{s_{j}}.\]
 Let $\sigma _{j}$ be an integer such that $\sigma_{j} \cdot \gamma_{j}-1$ is divisible by $p^{n}$. 
We define \[W(a)=\sum _{i=1}^{\infty }\sigma_{i}E_{j,a}(\upharpoonleft  ).\]
We also define \[W(a, b)=\sum _{i=1}^{\infty }\sigma_{i}E_{j,a}(b ).\]
 \end{definition}
 
Observe that \[W(\Omega (a), \upharpoonleft)=W(\Omega (a))=a.\]
 We are now ready to prove the main result of this section.

\begin{theorem}\label{23}
 Let $A$ be a pre-Lie ring satisfying property $3$.  Let $W$ 
 be defined as above, then the map $W:A\rightarrow A$ given by $a\rightarrow W(a)$ is injective. Therefore, there is a map $\Omega :A\rightarrow A$ such that 
$W(\Omega (a))=\Omega(W(a))=a$ for all $a\in A$. Also $\Omega =W^{{p^{n+1}!}-1}$. 
\end{theorem}
\begin{proof} The statement about  $\Omega $
 follows from the general fact about bijective functions on a finite set.  See for example Lemma \ref{g(a)}.
 
$ $

So it suffices to show that $W:A\rightarrow A$ is an injective function.
 Suppose on the contary that $W:A\rightarrow A$ is not injective. Without restricting generality we can assume that $A$ has the smallest possible cardinality such that $W:A\rightarrow A$ is not injective (and $A$ satisfies property $3$).

 Suppose that $p^{3}A=0$. 
  By property $3$, $A^{p-1}\subseteq p^{3}A=0$, so the Lazard's correspondence holds for the Lie ring associated to the pre-Lie ring $A$. Therefore, the function $W$ on pre-Lie ring $A$ is injective
 by Section $4$ \cite{passage}.
 Therefore we can assume that $p^{3}A\neq 0$. 

$ $

  Suppose that $W(a)=W(b)$ for some $a,b\in A$. Hence, $W([a]_{p^{3}A})=W([b])_{p^{3}A}$, and so $[a]_{p^{3}A}=[b]_{p^{3}A}$ in $A/p^{3}A$. Therefore, $a-b\in p^{3}A$. Let $j\geq 0$ be the maximal natural number such that $a-b\in p^{j+3}A$.
Let $c\in A$ be such that \[b-a=p^{j+3}c.\] 
Recall that $W(a)=W(b)$. Observe that  
\[W(a+p^{j+3}c)+(p^{j+2}-1)W(a)=p^{j+2}W(a+pc)+p^{j+4}e,\]
for some $e\in A$. 
Indeed terms in this expression which have at least two occurences of $c$ will have power 
$p^{j+4}$ near them, and 
terms which have only one $c$ would agree, as it would not be affected by $\rho ^{-1}$ reasoning as in  Remark \ref{moving} (since $p^{i-1}a-\rho ^{-1}(p^{i}a)\in ann(p)$ for $i>0$). 
 Moreover, terms without any $c$ will also agree.

Recall that $W(a)=W(b)=W(a+p^{j+3}c)$, it follows that
 \[p^{j+2}(W(a+pc)-W(a))= W(a+p^{j+3}c)-W(a)-p^{j+4}e=W(b)-W(a)- p^{j+4}e=-p^{j+4}e.\]
 Therefore 
$p^{j+2}(W(a+pc)-W(a)+p^{2}e)=0$.
 Therefore 
 \[W(a+pc)-W(a)\in p^{2}A+ann(p^{j+2}).\]
 Denote $I=p^{2}A+ann(p^{j+2})$, then 

\[W([a+pc]_{I})=W([a]_{I})\] in $A'=A/I$ (this follows from Property $3$ and since $W(a+pc)-W(a)\in I$ and from the fact that $ann(p^{j+2})\cap p^{i}A=p^{i}ann(p^{j+2+i})$ for each $i,j$).  
 Recall that $W:A'\rightarrow A'$ is injective (as at the beginning of this proof since $p^{3}A'=0$) so $[a+pc]_{I}=[a]_{I}$, hence $pc\in I$.
 Therefore, \[b-a=p^{j+3}c=p^{j+2}\cdot pc\in p^{j+2}(p^{2}A+ ann(p^{j+2}))\subseteq p^{j+4}A,\] 
 contradicting the maximality of $j$.
We have obtained a contradication. 
\end{proof}

\begin{definition}\label{gf} {\em The  modified construction of the group of flows.} Let $(A, +,\cdot )$ be a pre-Lie ring of cardinality $p^{n}$ for some prime number $p$ and some natural number $n$. Suppose that 
 $A$ satisfies property $3$ and let $\Omega (a)$ and $W(a)$, $W(a,b)$  be defined as in Definition \ref{e}.  
We say that $(A, +, \circ ')$ is obtained by the modified group of flows construction  from a pre-Lie ring $(A, +, \cdot  )$ if the operation $\circ ' $ in $A$ is defined as follows: 
\[a\circ ' b=a+b+W(\Omega (a), b).\]
 for $a,b\in A$.
 Moreover, the addition in $(A, +, \circ ')$ is the same as the addition in the pre-Lie ring $A$.  Denote $a\circ ' b=a+b+a*'b$, then 
\[a*'b=W(\Omega (a), b).\]
\end{definition}

We will now describe some properties of the modified construction of  group of flows.

Denote \[a^{\circ ' i}=a\circ '(a\circ '(\cdots a\circ ' a)).\]

\section{Generalised construction of group of flows and pseudobraces}

\begin{theorem}\label{w}
Let $(A, +, \cdot )$ be a pre-Lie ring which satisfies property $3$. Let $(A, +, \circ ')$ be  obtained from the pre-Lie ring  $(A, +, \cdot)$ by the generalised construction of the group of flows, as in Definition \ref{gf}. 
 Then $(A, +, \circ ')$ is a pseudobrace. 
\end{theorem}
\begin{proof}
Let $a,b\in A$, observe first that for every natural number $i$, we have 
\[W(i\cdot a)=W(a)^{\circ' i}.\]
 Moreover, \[W(i\cdot a)\circ ' b=W(a)\circ ' (W(a)\circ ' \cdots \circ ' (W(a)\circ ' b))\] where $W(a)$ appears $i$ times on the right hand side of this equation. 
 It can be proved by induction on $i$, in the same way as in the case of nomal group of flows, by using the fact that $\rho ^{-1}$ is similar to the division by $p$, and by using the Remark \ref{moving}. It is very similar to the proof that $e^{ia}\cdot b=(e^{a})^{i}\cdot b$ for real numbers $a,b$.
 
Observe that this implies
\[(W(a)^{\circ ' i}\circ ' W(a)^{\circ ' j})\circ b=W(a)^{\circ 'i+j}\circ ' b=W(a)\circ '(\cdots \circ '(W(a)\circ ' b)),\]
where $a$ appears $i$ times on the rght hand side for any $i$.

By Theorem \ref{23} function $W:A\rightarrow A$ is injective, and $W(\Omega (a)=a$ and $W(\Omega (b))=b$. Therefore we can apply the above to $\Omega (a)$ instead of $a$,  and $\Omega (b)$ instead of $b$ and obtain
 $(a^{\circ  i}\circ a^{\circ j})\circ b=a^{ \circ i+j}\circ b$ 
 and $a^{\circ i}\circ a^{\circ j}=a^{\circ i+j}$, $a^{circ i}\circ b=a\circ (a\circ \cdots a\circ b))$ where $a$ appears $i$ times on the right-hand side of this equation. 
 Observe also that $W(a)\circ W(-a)=0$ (where $0$ is the identity of the additive group of $A$).
This shows that $A$ satisfies some axioms of the pseudobrace from Definition \ref{pseudobrace}.

 We will now show that $A^{n+1}=0$ where $A^{1}=A$, and for $i\geq 1$ we define $A^{i+1}=A*'A^{i}$ where $a*'b=a\circ b-a-b$.

For the pre-Lie ring $A$ denote 
 $A(1)=A$ and  $A(i+1)=A\cdot A(i)$. Observe that by property $1'$ we have $A(m)\subseteq p^{n}A=0$ for some $A$ (where $p^{n}A=0$ by the Lagrange theorem for the additive group $(A,+)$).  
 
 Indeed, observe that  $A(i+1)\subseteq A(i)$ for each $i$, also  if $A(i)=A(i+1)$ for some $i$ then $A(i+1)=A(i+2)$ and so $A(i)=\cdots =A(m)=0$.  So it follows that $A^{n+1}=0$.

Note that by  $A*' A(i)\subseteq A^{i+1}$ by  Remark \ref{czesto} and by Remark \ref{moving}, therefore, $A^{n+1}\subseteq A(n+1)=0$.  
 Consequently $(A, +, \circ ')$ is a pseudobrace. 
\end{proof}

\section{Main result}
\subsection{Supporting lemmas}

 We start with a lemma.
\begin{lemma}\label{sp3}
Assume that  $A$ is a brace satisfying Property $1'$ and $1''$. Then  the pre-Lie ring  $(A/ann(p^{2k}), +,\bullet )$ constructed in Theorem \ref{preLie} 
 satisfies property $3$.
 Therefore the generalised group of flows  construction 
 from Definition \ref{gf} can be defined on this pre-Lie ring.
\end{lemma}
\begin{proof}
 This follows from Proposition $20$ in \cite{paper1}.
\end{proof}

\begin{lemma}\label{ideal}
 Let notation be as in Lemma \ref{sp3}. 
Let $(C, +, \circ )$ be a brace satisfying Properties $1'$ and $1''$, and let $C_{1}=(C/ann(p^{2k}), +, \circ )$ be its factor brace and $C_{2}=(C/ann(p^{2k}),+, \circ ' )$ be the pseudobrace obtained by the generalised group of flows construction  from the pre-Lie ring 
$(C/ann(p^{2k}), +,\bullet )$ as described in  Definition \ref{gf}. Denote $A=C/ann(p^{2k})$ as an abelian group. 

 Then $Ann(p^{j})=\{r\in A: p^{j}r=0\}$   is an ideal in the brace $C_{1}$
 (by \cite{paper1}) and in  the 
 pseudobrace $C_{2}$ for $j=1,2, \ldots $. 
\end{lemma}
\begin{proof}  We will show that $Ann(p^{j})$ is an ideal in pseudobrace $C_{2}$.
Indeed, let $a_{1}, \ldots , a_{\lfloor{\frac {p-1}4}\rfloor }\in C_{2}$ be such that some $a_{i}\in Ann(p^{j})$, denote 
 $e={a_{1}}\cdot ({a_{2}}\cdot (\cdots {a_{\lfloor{\frac {p-1}4}\rfloor }}))$. Then, by property $3$ we have $e\in pC_{2}\cap Ann(p^{j})$. Observe that $pC_{2}\cap Ann(p^{j})=p\cdot Ann(p^{j+1})$ hence  
$\rho ^{-1}(e)\in Ann(p)+pAnn(p^{j+1})$. Moreover, for 
$b_{1}, \ldots , b_{\lfloor{\frac {p-1}4}\rfloor }\in C_{2}$ we have then
 by property $3$:
 $b_{1}\cdot(\cdots  b_{\lfloor{\frac {p-1}4}\rfloor }\cdot \rho ^{-1}(e))\subseteq p\cdot Ann(p^{j+1})$.
 
 By the definition of $a*'b$ this argument applied several times implies that $\Omega (Ann(p^{2k}))\subseteq Ann(p^{i})$ and $Ann(p^{j})*'b\subseteq Ann(p^{j})$. 
  This follows from Remark \ref{moving}. Therefore $Ann(p^{j})*'b\subseteq Ann(p^{j})$. Notice that the definition of an ideal in a pseudobrace is slightly different than in brace, but 
the same reasoning also shows that $Ann(p^{j})$ is an ideal in $C_{2}$.
\end{proof}

\begin{lemma}\label{ideal2}
 Let notation be as in Lemma \ref{ideal}.
 Then $p^{j}A$   is an ideal in brace $C_{1}$
 (by \cite{paper1}) and in  
 pseudobrace $C_{2}$ for $j=1,2, \ldots $.  
\end{lemma}
\begin{proof} A similar proof as in Lemma \ref{ideal} works. 
\end{proof}

\begin{lemma}\label{property3}
  Let $p>3$ be a prime number, $n$ be a natural number and let $A$ be a pre-Lie ring of cardinality $p^{n}$.
 Let $A$ be a pre-Lie ring satysfying property $3$. Let $(A, +, \circ ')$ be the pseudobrace obtained from  the pre-Lie ring $(A, +, \cdot)$, by the generalised group of flows construction  from Definition \ref{gf}. 
  Then $(A, +, \circ ')$ satisfies Properties $1'$ and $1''$, and hence satisfies Property $1$.
\end{lemma}
\begin{proof}
 A similar proof as in Lemma \ref{ideal} works.
\end{proof}

\subsection{  Connections between $\odot $ and $\odot '$}

 The main result of this section is the following.

\begin{lemma}\label{c} Let notation and assumption be as in Lemma \ref{ideal}. 
 Consider the factor brace $A_{1}=C_{1}/Ann(p^{2k})$ and the factor pseudobrace $A_{2}=C_{2}/Ann(p^{2k})$. Then for $a,b\in A$ 
\[[a]_{Ann(p^{2k})}\odot [b]_{Ann(p^{2k})}=[a]_{Ann(p^{2k})}\odot '[b]_{Ann(p^{2k})}\] where 
  the operations $\odot $, $\odot '$ are defined for $a,b$ in the set $A$  by 
$[a]_{Ann(p^{2k})}\odot [b]_{Ann(p^{2k})}=[\wp^{-1}((p^{k}a)*b)]_{Ann(p^{2k})}$ and  $[a]_{Ann(p^{2k})}\odot' [b]_{Ann(p^{2k})}=[\wp^{-1}((p^{k}a)*'b)]_{Ann(p^{2k})}$. 
\end{lemma}
\begin{proof}   We denote $[a]=[a]_{ann(p^{2k})}$ for $a\in C$.
 
{\em Part $1$.} Observe that $p^{k}C_{2}$  is obtained by the construction of generalised group of flows from the pre-Lie ring $(p^{k}C/ ann(p^{2k}), +,\bullet )$.   Because pre-Lie ring $p^{k}C/ann(p^{2k})$ is strongly nilpotent on nilpotency index less than $p$, then the group of flows construction 
coincides with the construction of the group of flows.

Therefore $(p^{k}C/ann(p^{2k}), +, \circ ' )$ is obtained by the construction of group of flows from the pre-Lie ring $(p^{k}C/ ann(p^{2k}), +, \bullet )$.

{\em Part 2.} Observe on the other hand that $(p^{k}C/ann(p^{2k}), +, \circ )$ is a subbrace of brace $(C/ann(p^{2k}), +, \circ )$.
Recall that $A=C/ann(p^{2k})$ and  $(A, +, \circ)$ is a brace.

 Observe that brace $B=p^{k}A$ has strong nilpotency index less than $p$. 
By using the same proof as in \cite{passage} we see that  brace $B=p^{k}A$
  is obtained as a group of flows of a left nilpotent  pre-Lie ring $(B, +, \cdot )$. Because brace $B$ is strongly nilpotent of strong nilpotency order less than $p$ we can use 
the same proofs as in \cite{passage} (or we can prove this fact  by  using a more recent and more general result proved in \cite{Senne} that every brace of the left nilpotency index less than $p$ can be obtained as the group of flows of some left nilpotent pre-Lie algebra). We obtain that 
for $x,y\in B,$ 
\[\sum_{i=0}^{p-2}\xi ^{p-1-i}((\xi ^{i}x)* y)=(p-1)x\cdot y.\]
Therefore, for $x=p^{k}a$, $y=p^{k}b$ for $a,b\in C$ we have:
\[\sum_{i=0}^{p-2}\xi ^{p-1-i}((\xi ^{i}p^{k}a)* (p^{k}b))=(p-1)(p^{k}a)\cdot (p^{k}b).\]

For $a\in A$ denote $[a]=[a]_{ann(p^{2k})}$.
Then, by definition of $\bullet $ we have 
  \[p^{2k}[a]\bullet [ b]= [(p^{k'}a)\bullet (p^{k}b)]=[p^{k}a]\cdot [p^{k}b].\] 
 Hence, as the pre-Lie operation $\bullet $ is distributive with respect to addition we get
\[[p^{k}a]\bullet [p^{k}b]=[p^{k}a]\cdot [p^{k}b],\]
 for $a,b\in C$.
 Therefore, on $p^{k}C/ann(p^{2k})$ the operations $\bullet $ and $\cdot $ coincide. 
 Therefore, $(p^{k}C/ann(p^{2k}), +, \circ )$ is obtained by the formula of the group of flows from pre-Lie ring 
$(p^{k}C/ ann(p^{2k}), +, \bullet )$.

 This shows that  $(p^{k}a)*(p^{k}b)=(p^{k}a)* ' (p^{k}b)$ for $a,b\in A$.
 Observe that $(p^{k}a)*(p^{k}b)\in p^{2k}A$ since $C$ satisfies Property $2$ by Lemma \ref{property2'}, so $p^{k}A$ is an ideal in $C$ by definition of Property $2$.
 This is implies \[\wp^{-1}(\wp^{-1}((p^{k}a)*(p^{k}b)-(p^{k}a)* ' (p^{k}b)))\in Ann(p^{2k}).\]
 Therefore, 
$[a]_{Ann(p^{2k})}\odot [b]_{Ann(p^{2k})}=[a]_{Ann(p^{2k})}\odot '[b]_{Ann(p^{2k})}$ for all $a,b\in A$.

\end{proof}

 Our main result in this subsection is the following
\begin{lemma}\label{property2}
  Let $p>3$ be a prime number, $n$ be a natural number and let $A$ be a pre-Lie ring of cardinality $p^{n}$.
 Let $A$ be a pre-Lie ring satysfying property $3$.
 Let $(A, +, \circ ')$ be the pseudobrace obtained from  the pre-Lie ring $(A, +, \cdot)$, by the
 generalised group of flows construction  from Definition \ref{gf}.  Then $(A, +, \circ ')$ satisfies Property $2$. 
\end{lemma}
\begin{proof}
  By using a very similar  proof as in Lemma \ref{ideal}  we obtain that $ann(p^{i})$ and $p^{i}A$ are ideals 
 in $A$. Observe that the property $(p^{i}A)*ann(p^{j})\subseteq ann(p^{j-i})$ follows from the fact that $\Omega (p^{i}A)\subseteq p^{i}A$ and because the multiplication by $p^{i}A$ is distributive in pre-Lie rings (and then it follows by  using a similar proof as in Lemma \ref{ideal}).

It remains to show that 
  for each $a,b\in A$, 
 we have that $p^{2k}a=p^{2k}b$ if and only if $p^{k}a^{\circ ' p^{k}}=p^{k}b^{\circ ' p^{k}}$. 

Observe first that by Theorem \ref{23} we have that the map $W:A\rightarrow A$ is injective, so it suffices to show that $p^{2k}W(a)=p^{2k}W(b)$ if and only if 
$p^{k}W(a)^{\circ 'p^{k}}=p^{k}W(b)^{\circ ' p^{k}}$. 
 By Lemma \ref{w} the later assertion is equivalent to 
$p^{k}W(p^{k}a)=p^{k}W(p^{k}b)$. Observe that the pre-Lie ring $p^{k}A$ is strongly nilpotent 
 of strong nilpotency index less than $p$, so the construction of the group of flows from \cite{AG} agrees with
 the constrction of generalised group of flows defined in Definition \ref{gf}. 
 Therefore $(p^{k}A, +, \circ ')$ is a brace.  Therefore, we can denote $B=p^{k}A$ and restate our statement as to show that in the brace $B$ we have
that $p^{k}x=p^{k}y$ is equivalent to $p^{k}W(x)=p^{k}W(y)$ (where $x=p^{k}a, y=p^{k}b$).

 Observe first  that $p^{k}x=p^{k}y$ implies $p^{k}W(x)=p^{k}W(y)$ for $x,y\in B$. 
 Indeed, denote $y=x+c$ then we have $p^{k}c=0$. Also $p^{k}(W(x)-W(y))$ can be writen as a sum of elements of the form $p^{k}d$ where $d$ is a product of some elements from the set $\{x,c\}$ and $c$ appears at least once under the operation $\cdot$. The operation $\cdot $ is distributive with addition, so $p^{k}d$   is a product of some elements from the set $\{x,c, p^{k}c\}$ and $p^{k}c$ appears at least once. Therefore, $p^{k}(W(x)-W(y))=0$.  This shows that $p^{k}x=p^{k}y$ implies $p^{k}W(x)=p^{k}W(y)$ for $x,y\in B$.

 We will now show that $p^{k}W(x)=p^{k}W(y)$ implies $p^{k}x=p^{k}y$. 
 Observe that $p^{k}W(x)=p^{k}W(y)$ is equivalent to  $W(x)-W(y)\in ann(p^{k}).$ Therefore, it suffices to show that
 $[W(a)]_{ann (p^{k})}=[W(b)]_{ann(p^{k})}$
 implies $[a]_{ann(p^{k})}=[b]_{ann(p^{k})}$ where $ann(p^{k})=\{a\in B:p^{k}a=0\}$. Observe that function $W:B/p^{k}B\rightarrow B/p^{k}B$ is injective because $B/p^{k}B$ satisfies property $3$, since $A$ satisfies property $3$ (by Theorem \ref{23}). Therefore,  
$[W(a)]_{ann (p^{k})}=[W(b)]_{ann(p^{k})}$
 implies $[a]_{ann(p^{k})}=[b]_{ann(p^{k})}$, as required.   
\end{proof}
The following Lemma immediately follows from \cite{paper1}.
\begin{lemma}\label{property2'}
Let $A$ be a brace satisfying Property $1$. Then $A$ satisfies property $2$.
\end{lemma}
\begin{proof}
Theorem $8$ and Lemma $10$ from \cite{paper1}.
\end{proof}

\subsection{ Factor braces are groups of flows}
In the statement of the following theorem we use Lemmas \ref{sp3} and \ref{ideal} to see that the constructions are well defined. Let $C, C_{1}, C_{2}, A$ be as in Theorem \ref{main1} below. 
  By Lemma \ref{ideal}  $Ann(p^{2k})$ is an ideal in the brace $C_{1}$ and also in the pseudobrace $C_{2}$ where  
 $Ann(p^{2})= \{a\in A:p^{2k}a=0\}$.

\begin{theorem}\label{main1}
 Let $(C, +, \circ )$ be a brace satisfying properties $1'$ and $1''$  and let $C_{1}=(C/ann(p^{2k}), +, \circ )$ be the factor brace where $ann(p^{2k})=\{a\in C:p^{2k}a=0\}$. 
 Let $k$ be such that $p^{k(p-1)}C=0$. 
 Let $C_{2}=(C/ann (p^{2k}), +, \circ ')$ be obtained by the construction of modified
 group of flows from the pre-Lie ring $(C/ann(p^{2k}),+, \bullet )$ defined in Theorem \ref{preLie} for brace $C$.
 Denote $A=C/ann(p^{2k})$ as an abelian group, and $Ann(p^{k})=\{a\in A: p^{2k}a=0\}$. Then $C_{1}=(A, +, \circ )$, $C_{2}=(A, +, \circ ')$. 
 Then the factor  brace $C_{1}/ Ann(p^{2k})$ equals the factor pseudobrace $C_{2}/ Ann(p^{2k})$. 
 Therefore, the factor brace  $C_{1}/ann(p^{2k})$ is obtained by the construction of modified group of flows from the pre-Lie ring 
$C_{2}/ann(p^{2k})$. 
\end{theorem}
\begin{proof}
\begin{enumerate}
\item By Proposition $20$  from \cite{paper1} the pre-Lie ring $(C/ann(p^{2k}),+, \bullet )$ satisfies property $3$. 

\item By assumptions of this theorem  brace $A_{1}$ satisfies properties $1'$ and $1''$, and by Lemma \ref{property3}  pseudobrace $A_{2}$ satisfies properties $1'$ and $1''$. Therefore, $A_{1}$ and $A_{2}$ satisfy property $1$. 
\item By  Lemma \ref{property2'} pseudobrace $A_{1}$ satisfies property $2$, and by Lemma \ref{property2} brace $A_{2}$ satisfies property $2$. 
\item It was showed in Lemma \ref{c}  that the  operation $\odot $ agrees on brace  $(A/Ann(p^{2k}), +, \circ )$ and on the pseudobrace  $(A/Ann(p^{2k}), +, \circ ')$.
\item  By Theorem \ref{71} applied to brace $A_{1}=(A/Ann(p^{2k}),+, \circ )$ and pseudobrace $A_{2}=(A/Ann(p^{2k}), +, \circ ')$   we obtain that the operation $*$ is the same on brace $A_{1}$ and pseudobrace $A_{2}$, because it is obtained by using operations $\odot , +, \rho ^{-1}$,
 and the formula is the same for $A_{1}$ and $A_{2}$. So $[a]_{Ann(p^{2k})}\circ [b]_{Ann(p^{2k}}=[a]_{Ann(p^{2k})}\circ '[b]_{Ann(p^{2k})}$ for $a,b\in A$.

\item  Therefore,  brace $A_{1}$ and  pseudobrace $A_{2}$ are the same. Since the pseudobrace $C_{2}$ is obtained as a generalised group of flows
 from pre-Lie ring $(C, +, \bullet )$ then 
$A_{2}$ is obtained be the construction of generalised group of flows from factor pre-Lie ring 
$(C, +,\bullet )$  (it can be seen by reasoning similarly as in Lemma \ref{ideal}).

 \item Consequently brace  $A/ann(p^{2k})$ is obtained as a generalised group of flows. Note that the brace  $C/ann(p^{4k})$ is isomorphic to brace  $A/Ann(p^{2k})$.       
\end{enumerate}
\end{proof}
We get the following result
\begin{theorem}\label{main}
Let $(C, \circ , +)$ be a brace satisfying Property $1'$ and Property $1''$.
 Let $k$ be such that $p^{k(p-1)}C=0$. 
  Define for $a,\in C$:
\[[a]_{ann(p^{4})}\cdot [b]_{ann(p^{4k})}=-(1+p+p^{2}+\ldots +p^{n})\sum_{i=0}^{p-2}\xi ^{p-1-i}((\xi ^{i}p^{k}a)* b)]_{ann(p^{4k})}.\]
 Then $(C/ann(p^{4}), \cdot , +)$ is a pre-Lie ring, and the brace 
 $(C/ann(p^{4}), \cdot , +)$ is obtained by the construction of modified group of flows (from Definition \ref{gf}) 
 applied to the pre-Lie ring $(C/ann(p^{4}), \cdot , +)$.
\end{theorem}
\begin{proof} It follows from Theorem \ref{main1} and from the last sentence in the proof of Theorem \ref{main1}.
\end{proof}

\begin{theorem}\label{rank1}

Let $(A, \circ , +)$ be a brace of cardinality $p^{n}$ for a prime number $p$ and for a natural number $n$.
 Suppose that the additive group $(A, +)$ has rank not exceeding ${\frac {p-1}4}$. 
  Let $k$ be such that $p^{k(p-1)}A=0$. Define 
\[[a]_{ann(p^{4})}\cdot [b]_{ann(p^{4k})}=-(1+p+p^{2}+\ldots +p^{n})\sum_{i=0}^{p-2}\xi ^{p-1-i}((\xi ^{i}p^{k}a)* b)]_{ann(p^{4k})}.\]
 Then $(A/ann(p^{4}), \cdot , +)$ is a pre-Lie ring, and the brace 
 $(A/ann(p^{4}), \cdot , +)$ is obtained by the construction of modified group of flows 
 applied to this pre-Lie ring,  provided that this pre-Lie ring satisfies assumptions of Definition \ref{gf}.
\end{theorem}
\begin{proof}  The fact that $(A, +, \cdot)$ is a left nilpotent pre-Lie ring follows from Theorem \ref{preLie}. The result now follows from Theorem \ref{main} and Lemma \ref{r}.
\end{proof}
\begin{theorem}\label{uniform1}
Let $(A, \circ , +)$ be a brace of cardinality $p^{n}$ for a prime number $p$ and for a natural number $n$, whose additive group $(A, +)$ is uniform. Suppose that $A^{\lfloor \frac {p-1}4\rfloor}=0$ for all $a,b\in A$. Let $k$ be such that $p^{k(p-1)}A=0$. Define 
\[[a]_{ann(p^{4})}\cdot [b]_{ann(p^{4k})}=-(1+p+p^{2}+\ldots +p^{n})\sum_{i=0}^{p-2}\xi ^{p-1-i}((\xi ^{i}p^{k}a)* b)]_{ann(p^{4k})}.\]
 Then $(A/ann(p^{4}), +, \cdot )$ is a left nilpotent pre-Lie ring, and the brace 
 $(A/ann(p^{4}), +, \circ  )$ is obtained by the construction of modified group of flows 
 applied to this pre-Lie ring.
\end{theorem}
\begin{proof} By assumptions $A$ satisfies property $1'$. Observe that $A$ satisfies property $1''$ because 
  the additive group of $A$ is a direct sum of cyclic 
 groups of cardinality $p^{\alpha }$ for some $\alpha $. Observe that $A$ satisfies property $1''$ because  $ann(p^{i})=p^{\alpha -i}A$, for each $i$. 
 The result now follows from Theorem \ref{main}.
\end{proof}
 {\bf Proof of Theorems \ref{rank} and \ref{uniform}}: 
 It follows from Theorem \ref{main} and Theorems \ref{rank1}, \ref{uniform1}.

{\bf Proof of Theorem \ref{main2}.} It follows from Theorem \ref{main1}.

$ $

{\bf Acknowledgments.} The author acknowledges support from the
EPSRC programme grant EP/R034826/1 and from the EPSRC research grant EP/V008129/1. 
 The author is very greateful to Bettina Eick and Efim Zelmanov for answering her questions about what is known about extensions of Lazard's correspondence in group theory.

\end{document}